\documentclass{amsart}
\usepackage{amsmath,amsfonts,amsthm,amssymb}
\usepackage{color,soul}
\usepackage{cancel}
\usepackage[normalem]{ulem}
\usepackage{tikz}
\usetikzlibrary{arrows,decorations.pathmorphing,backgrounds,positioning,fit}
\usetikzlibrary{calc}
\usepackage{setspace}
\usepackage{caption}
\usepackage{subcaption}
\usepackage{url}

\newtheorem{thm}{Theorem}[section]

\newtheorem{lemma}[thm]{Lemma}

\newtheorem{coro}[thm]{Corollary}
\newtheorem{prop}[thm]{Proposition}
\theoremstyle{definition}
\newtheorem{defn}[thm]{Definition}
\newtheorem{notation}[thm]{Notation}

\theoremstyle{remark}
\numberwithin{equation}{section}

\newenvironment{ex}{\refstepcounter{thm}\begin{proof}[Example \emph{\thethm}]}{\end{proof}}
\newenvironment{rem}{\refstepcounter{thm}\begin{proof}[Remark \emph{\thethm}]}{\end{proof}}
\numberwithin{equation}{section}

\def\RR{\mathbb{R}}
\def\ZZ{\mathbb{Z}}

\def\Span{\mathrm{Span}}
\def\M{\mathbb{A}}
\newcommand{\garrow}[1]{\stackrel{#1}{\longrightarrow}}

\colorlet{gpurple}{red!35!blue}
\colorlet{ggreen}{green!50!black}

\newcommand{\ThreeD}[3]{($ #1*(1,0) + #2*(-.9,-.436) + #3*(-.5,.866) - .5*#3*(-.9,-.436)$)} 
\colorlet{dyellow}{yellow!50!black}

\newcommand{\ychamber}[1][(0,0)]{
\begin{scope}[shift={#1},opacity=.5]
	\node[coordinate] (000) at \ThreeD{0}{0}{0} {};
	\node[coordinate] (001) at \ThreeD{0}{0}{1} {};
	\node[coordinate] (101) at \ThreeD{1}{0}{1} {};
	\node[coordinate] (011) at \ThreeD{0}{1}{1} {};
	\node[coordinate] (111) at \ThreeD{1}{1}{1} {};
	\node[coordinate] (112) at \ThreeD{1}{1}{2} {};
	\draw[dyellow,fill=yellow!25] (000) to (101) to (112) to (011) to (000);
	\draw[dyellow!50] (101) to (001) to (011);
	\draw[dyellow!50] (000) to (001) to (112);
	\draw[dyellow] (101) to (111) to (011);
	\draw[dyellow] (000) to (111) to (112);
\end{scope}
}
\newcommand{\bchamber}[1][(0,0)]{
\begin{scope}[shift={#1},opacity=.5]
	\node[coordinate] (011) at \ThreeD{0}{1}{1} {};
	\node[coordinate] (111) at \ThreeD{1}{1}{1} {};
	\node[coordinate] (112) at \ThreeD{1}{1}{2} {};
	\node[coordinate] (122) at \ThreeD{1}{2}{2} {};
	\draw[blue,fill=blue!25,opacity=.5] (011) to (111) to (112) to (122) to (011);
	\draw[blue!50] (011) to (112);
	\draw[blue] (111) to (122);
\end{scope}
}
\newcommand{\rchamber}[1][(0,0)]{
\begin{scope}[shift={#1},opacity=.5]
	\node[coordinate] (101) at \ThreeD{1}{0}{1} {};
	\node[coordinate] (111) at \ThreeD{1}{1}{1} {};
	\node[coordinate] (112) at \ThreeD{1}{1}{2} {};
	\node[coordinate] (212) at \ThreeD{2}{1}{2} {};
	\draw[red,fill=red!25,opacity=.5] (101) to (111) to (112) to (212) to (101);
	\draw[red!50] (101) to (112);
	\draw[red] (111) to (212);
\end{scope}
}

\tikzstyle{mutable}=[inner sep=0.5mm,circle,draw,minimum size=2mm]
\tikzstyle{frozen}=[inner sep=.9mm,rectangle,draw]
\tikzstyle{dot} = [fill=black!50,inner sep=0.25mm,circle,draw,minimum size=1mm]
\tikzstyle{faint dot} = [draw=black!20,fill=black!10,inner sep=0.25mm,circle,minimum size=1mm]
\tikzstyle{blue dot} = [draw=blue,fill=blue!50,inner sep=0.25mm,circle,minimum size=1mm]
\tikzstyle{red dot} = [draw=red,fill=red!50,inner sep=0.25mm,circle,minimum size=1mm]
\tikzstyle{marked}=[inner sep=0.5mm,circle,draw,blue!75!black,fill=blue!50]

\tikzstyle{sgreen dot} = [draw=ggreen,fill=ggreen!50,inner sep=0.15mm,circle,minimum size=1mm]
\tikzstyle{sblue dot} = [draw=blue,fill=blue!50,inner sep=0.15mm,circle,minimum size=1mm]
\tikzstyle{sred dot} = [draw=red,fill=red!50,inner sep=0.15mm,circle,minimum size=1mm]

\def\Hom{\mathrm{Hom}}
\def\End{\mathrm{End}}
\def\gldim{\mathrm{gl.dim}}
\def\rgldim{\mathrm{r.gl.dim}}
\def\lgldim{\mathrm{l.gl.dim}}

\DeclareMathOperator{\Spec}{Spec}

\title{Non-Commutative Resolutions of Toric Varieties}  

\author{Eleonore Faber}
\address{School of Mathematics, University of Leeds, Leeds, LS2 9JT, UK}
\email{e.m.faber@leeds.ac.uk}

\author{Greg Muller}
\address{Department of Mathematics, University of Oklahoma, Norman, OK 73019, USA}
\email{gmuller@math.ou.edu}

\author{Karen E. Smith}
\address{Department of Mathematics, University of Michigan, Ann Arbor, MI 48109, USA }
\email{kesmith@umich.edu}

\begin{document}

\thanks{ The first author was funded by the European Union's Horizon 2020 research and innovation programme under the Marie Sk{\l}odowska-Curie grant agreement No 789580.
The third author was partially supported by the US National Science Foundation, grant DMS-1501625.
} 
\date{\today}

\maketitle

\begin{abstract}
Let $R$ be the coordinate ring of an affine toric variety. We prove, using direct elementary methods, that the  endomorphism ring 
 $\End_R(\mathbb A),$ where $\mathbb A$ is the  (finite) direct sum of all (isomorphism classes of) conic $R$-modules, 
  has finite global dimension equal to the dimension of $R$. 
  This gives a precise version, and an elementary proof, of a theorem of \v{S}penko and Van den Bergh implying that $\End_R(\mathbb A)$ has finite global dimension.
    Furthermore, we show that $\End_R(\mathbb A)$ is a non-commutative {crepant}  resolution if and only if the toric variety is simplicial.
For toric varieties over a perfect field $\mathsf{k}$ of prime characteristic, we show that the ring of differential operators $D_\mathsf{k}(R)$ has finite global dimension. \end{abstract}
\section{Introduction}


Consider a local or graded ring $R$ which is commutative and Noetherian. A well-known theorem of Auslander-Buchsbaum and Serre states that $R$ is regular if and only if $R$ has finite global dimension---that is, if and only if every $R$-module has finite projective dimension.

While the standard definition of regularity does not extend to non-commutative rings, the definition of global dimension does, so this suggests that finite global dimension might play the role of regularity for non-commutative rings. This is an old idea going back at least to Dixmier \cite{Dix63}, though since then our understanding of connection between regularity and finite global dimension has been refined by the works of Auslander, Artin, Shelter, Van den Bergh and others.  
One particularly relevant connection here is the following: {\it If $R$ is a regular ring of finite type over a perfect field $\mathsf{k}$, then the ring of differential operators $D_\mathsf{k}(R)$ has finite global dimension.}\footnote{We provide a counterexample to the converse in Theorem \ref{Dmod}.} This  theorem  is due  to Roos  in characteristic zero \cite{Roos, Chase}, and Paul Smith in characteristic $p$  \cite{PSmithCharpRegular}.

In this  paper, we establish the finite global dimension of several non-commutative algebras associated to a \emph{normal toric algebra $R$} over a field $\mathsf{k}$, a type of algebra which is commutative but rarely regular. These algebras are:
\begin{enumerate}
	\item $\mathrm{End}_R(\M)$, where $\M$ is a \emph{complete sum of conic modules}. `Conic modules' are certain combinatorial defined $R$-modules first studied systematically by Bruns and Gubaladze \cite{BrunsGubeladze},  and a direct sum of conic modules is `complete' if every conic module is isomorphic to a summand.

	\item $\mathrm{End}_R(R^{1/q})$, assuming $q$ is sufficiently large. Here, $R^{1/q}$ is the $\mathsf{k}$-algebra spanned by $q$-th roots of monomials in $R$. When $\mathsf{k}$ is perfect and $q$ is a power of the characteristic of $\mathsf{k}$, $R^{1/q}$ is the ring of $q$-th roots of all elements in $R$.
	\item The ring of $\mathsf{k}$-linear differential operators $D_\mathsf{k}(R)$ of $R$, assuming $\mathsf{k}$ is perfect with positive characteristic. 
\end{enumerate}
Our  primary result is an explicit projective resolution for each simple module for the algebra $\mathrm{End}_R(\M) $ in (1), from which it follows  that the global dimension is equal to the dimension of $R$; see Theorem \ref{gldim}.  Similar results follow for the algebras $\mathrm{End}_R(R^{1/q})$ in case (2), which we show in Proposition \ref{ISOM} are Morita equivalent to $\mathrm{End}_R(\M)$. 
The result for differential operators  follows by taking  a direct limit; see Theorem \ref{Dmod}. Our proofs are elementary,   direct, self-contained, and constructive.  It should be noted that in a very nice paper, \v{S}penko and Van den Bergh independently found less constructive proofs for the finiteness of global dimension in settings (1) and (2), albeit without explicit bounds on the global dimension \cite[1.3.6]{SVdB2}; they also beat us to publication. See  also  \cite{RSvdB3}.

Because these algebras have finite global dimension, they should be regarded as `regular' in an appropriate non-commutative sense. Concretely, we may say the first two are \textbf{non-commutative resolutions} of the toric algebra $R$,  as defined by Bondal and Orlov
 \cite{BondalOrlov}; see also \cite{Leuschke, Wemyss, NCR}.  
 
 Finally, we consider when these non-commutative resolutions are \textbf{crepant}, a condition introduced in \cite{VanDenBergh} which ensures the module category of the resolution is sufficiently close to the module category of the original algebra. We conclude that $\mathrm{End}_R(\M)$ is a non-commutative crepant resolution if and only if the toric algebra $R$ is \emph{simplicial} (Theorem \ref{simplicial}).

This work naturally leads to several questions for future research.
\begin{itemize}
	\item What conditions on a ring $R$ ensure that  $D_\mathsf{k}(R)$ has finite global dimension?
	\item For an arbitrary algebra $R$ of characteristic $p$, what conditions on $R$ imply the global dimension of $\mathrm{End}_R(R^{1/p^e})$ is finite for sufficiently large $e$?
	\item When is  there some (non-complete) sum  $\mathbb B$ of conic modules over a toric ring such that the endomorphism ring  $\mathrm{End}_R(\mathbb B)$ is a non-commutative crepant resolution?
	\item  Can our work on the finite global dimension of $D_\mathsf{k}(R)$ in prime characteristic be used to prove  finite global dimension of   $D_\mathsf{k}(R)$ in characteristic zero  (when $R$ is a toric algebra)? Recent work of Jeffries \cite{Jeff17} shows that the obstruction of reduction to prime characteristic for differential operators introduced in   \cite{SmVdB} may vanish for all toric algebras. See also \cite{MussonVdB} for a more direct approach to showing that rings of differential operators on toric rings of characteristic zero have finite global dimension.
	\end{itemize}

\subsection{The proof: Conic modules of toric algebras}

The preceding results are proven by carefully considering the \emph{conic modules} of a normal toric algebra $R$. These are fractional monomial ideals of $R$ whose monomial support is given by shifting the defining cone of $R$.
Conic modules are rank one and Cohen-Macaulay, and each toric $R$ has finitely many conic modules up to isomorphism. Conic modules may be parameterized by polyhedral \textbf{chambers of constancy} inside a hyperplane arrangement.  These properties of conic modules have been systematically studied by Bruns and Gubeladze \cite{BrunsGubeladze, Bruns/arXiv:math/0408110}.

For each chamber of constancy $\Delta$, we use the combinatorics of the faces in $\Delta$ to construct a chain complex $K_\Delta^\bullet$ of sums of conic modules. The essential property of these complexes is the Acyclicity Lemma (Lemma \ref{acyclicity}), which states that the complex $\mathrm{Hom}_R(A_{\Delta'},K_\Delta^\bullet)$ is either acyclic or a resolution of the ground field $\mathsf{k}$, depending on whether or not $A_\Delta\simeq A_{\Delta'}$.

The benefit of this lemma is as follows. A direct sum of conic modules $\M$ is \emph{complete} if every conic module of $R$ is isomorphic to a summand of $\M$. Our motivating example is the ring $R^{1/n}$ spanned by formal $n$th roots of monomials in $R$, which is a complete sum of conic modules provided $n$ is large enough (Proposition \ref{ISOM}).

Given a complete sum of conic modules $\M$ and a chamber of constancy $\Delta$, the Acyclicity Lemma implies that the complex $\mathrm{Hom}_R(\M,K_\Delta^\bullet)$ is a finite projective resolution of a simple $\mathrm{End}_R(\M)$-module. Using standard arguments on global dimension (and a nice choice of grading),  we can show that every finitely generated $\mathrm{End}_R(\M)$-module has a finite projective resolution; and thus, the finiteness of the global dimension.

\subsection{Relation with the Frobenius map}
The classification of commutative rings of characteristic $p$ according to the behavior of Frobenius has a long tradition in commutative algebra.
 This began with Kunz's theorem that regularity can be characterized by the flatness of Frobenius  \cite{Kunz}, which led to the development of tight closure and F-regularity  \cite{HH90}. Eventually,  relationships with the singularities in  the minimal model program  were discovered and developed  \cite{Smi97b}, \cite{SmithGloballyFRegular}, 
\cite{HaraWatanabe}, \cite{Takagi}, as well as connections with rings of differential operators  \cite{SmithDMod} \cite{SmVdB}. Since then, deeper connections with the minimal model program, including 
 the log canonical threshold, multiplier ideals, and log Fano varieties  (see  for example,  \cite{Schwede} \cite{HaraYoshida}, \cite{ST11},  \cite{Blickle}, \cite{Hacon}),  differential operators  (see, for example \cite{JeffriesNunez}), and F-signature 
 (see \cite{HunekeLeuschke},  \cite{Tucker}, \cite{C-RST}, for example) have blossomed.

Our work suggests that under some hypothesis,  Frobenius may provide  a way to construct a {\it functorial}  non-commutative resolution of singularities. The property that a $\mathbb C$-algebra admits a 
  non-commutative crepant resolution is closely related to rational singularities (see \cite{VanDenBergh, StaffordVanden}). On the other hand,  the property of strong-F-regularity is a prime characteristic analog of rational singularities (at least in the Gorenstein case).
We speculate that for  a strongly F-regular ring $R$ with finite F-representative type (a class of rings including toric rings in prime characteristic),  perhaps  $\End_R(R^{1/p^e})$ always has finite global dimension.   This speculation encompasses the  classes of rings arising in representation theory where this idea has been explored from a somewhat different point of view in \cite{SVdB2,RSvdB3}.

Our work corrects a result of Yasuda  \cite{Yasuda}, and is closely related to recent work of Higashitani and Nakajima  \cite{HibiRings}.  A different approach to the existence of non-commutative crepant resolutions for three dimensional Gorenstein toric rings  in low dimension can be found in  \cite{Broomhead}; see also \cite{BocklandtNCCR}.
As mentioned earlier, our results overlap  considerably with work of 
 Van den Bergh and {\v S}penko, although their machinery is quite different and less constructive;  see \cite{SVdB1}, 
\cite{SVdB2}.
 Our technique, by contrast, is a fairly direct and self-contained, so which should be widely accessible, including to commutative algebraists and algebraic geometers.
    
\subsection{Outline of the paper}

After recalling basic definitions and notation associated with toric algebras (also known as normal semi-group rings) in Section 2, we recall the notion  {\it conic modules}  from \cite{BrunsGubeladze} and  develop it for our purposes in Section 3, culminating in Brun's cell-decomposition of $\mathbb R^d$ induced by any full dimensional pointed rational polyhedral cone $C$ in $\mathbb R^d$  \cite{Bruns/arXiv:math/0408110}. Also in Sections 3 and 4, we  introduce the idea of a {\it open conic module} and its relationship with the cell decomposition,  and an instrumental idea in the following sections, and prove several other needed results, for example on the behavior of conic modules under restriction.   The technical heart of the paper is in Section 5, where we  use this cell decomposition to construct complexes of conic modules over $R$ and prove the crucial Acyclicity Lemma, Lemma  \ref{acyclicity}. The benefits of this work is reaped in Section 6, where  we show how to  induce explicit projective resolutions of simple modules over $\End_R(\mathbb A)$. Also in Section 6, we finish the proof of Theorem \ref{gldim} that $\End_R(\mathbb A)$ has finite global dimension, and deduce 
Corollary   \ref{Dmod} stating that  in prime characteristic,  rings of differential operators on toric varieties have finite global dimension. Finally in Section 7, we prove Theorem \ref{simplicial} characterizing the non-commutative resolution $\End_R(\mathbb A)$ as {crepant} if and only if $R$ is simplicial. 

\def\k{\mathsf{k}}
\def\div{\mathrm{div}}

\section{Toric Algebras}\label{Notation}
 Fix an arbitrary ground field $\k$. 
We begin with a recap of  toric algebras over $\k$, otherwise known as finitely generated normal semi-group algebras.
 See \cite{Fulton} for details. 

Let $M$ be a finitely generated free abelian group (also called a {\it lattice}), and let  $M_\mathbb{R}:= \mathbb{R}\otimes M$ be the associated real finite dimensional vector space. 
A {\bf rational polyhedral cone } $C$ in $M_{\mathbb R}$ is a  closed set consisting of all  non-negative real linear combinations of a finite set of elements in $M$.

Given a rational polyhedral cone $C\subset M_\mathbb{R}$, the associated \textbf{toric algebra} $R = \k[C\cap M] $   is the $\k$-subalgebra of the Laurent ring $\k[M] = \k[x^m\, | \, m \in M]$ 
defined by 
$$R := \Span \{x^m  \, | \, 
m\in M\cap C\}.$$
 Alternatively, the ring $R$ can be defined simply as the 
 semi-group algebra of the semi-group $C\cap M$. 
The Laurent ring   itself is  the special case where $C = M_{\mathbb R}$.

The algebra $R$ is a domain   finitely generated over $ \k$,  naturally graded by the group $M$; its  dimension is  equal to the rank   of the lattice $M$ 
\cite[\S 1.2]{Fulton}. It is also well-known to be normal and Cohen-Macaulay \cite[p29-30]{Fulton}.

The algebra $R = \k[C\cap M]$  is the coordinate ring of the associated {\bf affine toric variety} $$X_{\sigma} = \, {\text{Spec}} \, R,  $$  where $\sigma \subset N_{\mathbb R} $ is the dual cone of  $C$ in the dual vector space $N_{\mathbb R} = M_{\mathbb R}^*$.  Typically in the literature, one begins with a finite rank lattice $N$, and 
a  rational polyhedral  cone $\sigma \subset N \otimes \mathbb R = N_{\mathbb R}$.  Our cone $C$ is the dual cone  to $\sigma$, namely
\begin{equation}\label{ring}
C = \{v \in M_{\mathbb R}\, | \, \langle v, w \rangle \geq 0\,\,\,{\text{for all }} w\in \sigma\},
\end{equation}
where $\langle v, w \rangle$ denotes the natural pairing between elements of $M_{\RR}$ and its dual space $N_{\mathbb R}$ and $M$ is the dual lattice $\Hom_{\mathbb Z}(N, \mathbb Z)$.

Without loss of generality, we can and will assume throughout that our cone $C$ is {\bf full dimensional} in $M_{\mathbb R}$. Indeed, we can  replace $M_{\mathbb R}$ with the span of $C$ if necessary, and the lattice $M$ by its intersection with this span.  
The assumption that $C$ is full-dimensional is equivalent to the assumption that the dual cone $\sigma$ is  {\bf pointed} (or {\bf strongly convex}), meaning that it contains no subspace of $N_{\mathbb R}$. In this case, it is easy that  $R$ and $\k[M]$ have the same fraction field,  so that the Krull dimension of $R$ is  equal to the rank of $ M$.

\subsection{Primitive Generators.} Fix a rational polyhedral cone $\sigma$ in $N_{\mathbb R} = N \otimes_{\mathbb Z} \mathbb R$.  Throughout this paper, the
 notation 
$$\Sigma_1 = : \{n_1, n_2, \dots, n_t\} \subset N$$ denotes a set of {\bf minimal primitive}  generators for $\sigma$.  
Minimality means that  none of the $n_i$ is redundant: removing any one will span a strictly smaller cone; equivalently in the case where $\sigma$ is pointed, this means that each $n_i$ spans an {\bf extremal ray} of $\sigma$ in $N_{\mathbb R}$. 
{\bf Primitivity}   means that  $n_i\in N$ but  $\varepsilon \,n_i \notin N$ for
$0 < \varepsilon  < 1$.   
If  $\sigma$ is pointed,  the set $\Sigma_1$ of minimal  primitive generators is uniquely determined.  

Equation (\ref{ring}) can be rewritten
\begin{equation}
\begin{aligned}
C &= \{v \in M_{\mathbb R}\, | \, \langle v, n_i \rangle \geq 0\,\,\,{\text{for all }} i = 1, \dots, t\} \\ &=
\label{intersection} \bigcap_{i=1}^t  \{v \in M_{\mathbb R}\, | \, \langle v, n_i \rangle \geq 0\},
\end{aligned}
\end{equation} 
which 
represents   $C$ as a finite intersection of closed half-spaces in $M_{\mathbb R}$. In particular, the {\bf facets}   (or codimension 1 subcones)  of $C$ are indexed  by the elements $n_i$ of $\Sigma_1$: each facet of $C$ is precisely the intersection  of $C$ with a supporting  hyperplane defined by the equation $\langle x, n_i \rangle = 0$ in 
$M_{\mathbb R}$. Abusing terminology slightly, we call the $n_i$ the {\bf  primitive inward-pointing  normals}   to the facets of $C$.

\subsection{Simplicial Toric Rings.}
Fix a toric ring $R = \k[C\cap M]$, where, without loss of generality, we assume that $C$ spans $M_{\mathbb R}$.  Recall that $R$ is said to be {\bf simplicial} if the primitive generators $\Sigma_1$  are {\it linearly independent.} Simplicial toric varieties are an especially nice class: for example, they have cyclic quotient singularities  \cite[2.2]{Fulton}.

Toric rings of dimension two are always simplicial, since a pointed full dimensional cone $\sigma$  in a two dimensional vector space has exactly two extremal rays. However, in higher dimension, the typical toric ring is not simplicial; see  Example \ref{ConeOverSquare}.

\section{Conic modules}
\label{sec: conic}

The monomials of a toric algebra $R$ are defined by intersecting a rational polyhedral cone $C$ with a lattice $M$.
If we translate the cone and intersect with $M$, we instead define the monomials of an $R$-module, called a \emph{conic module}. 

For the formal definition, fix a toric algebra $R$ corresponding to a full-dimensional rational polyhedral cone $C \subset M_{\RR}$. 

\begin{defn}\label{ConicModDef}
Fix a vector  $v\in M_\mathbb{R}$.  The {\bf conic module{\footnote{The name ``conic module" seems to have been coined in 
 \cite{BrunsGubeladze}, although these were studied earlier, for example, in \cite{Stanley, Dong}.}}
 defined by} $v$  is the $M$-graded $R$-submodule of the Laurent ring $\k[M]$ defined as follows:
\[ A_v := \Span \{x^m \mid m\in M\cap (C+v)\}.
\]\end{defn}

\begin{prop}\label{ConicModProp}
Let $A_v$ be a conic module for the toric algebra $R = \k[C\cap M]$. Then
\begin{enumerate}
\item
 $A_v$ is  torsion free and  rank one over $R$. 
\item $A_v$ is spanned by monomials $x^m$ where $m\in M$ is in the intersection of half spaces
$$
\bigcap_{n_i \in \Sigma_1} \{x\in M_{\mathbb R} \, | \,  \langle  x, n_i \rangle \geq   \lceil \langle  v, n_i \rangle \rceil \},
$$
where the notation $ \lceil \lambda \rceil$ denote the smallest integer $\geq \lambda$.
\item  For any two conic modules $A_v$ and $A_w$, the module $\Hom_{R}(A_v, A_w)$ is naturally isomorphic to the $M$-graded $R$-submodule of $\k[M]$
$$
\Span \{x^m\, | \, m+(C+v) \cap M \subset (C+w) \}.
$$
In particular, 
the only degree zero  $R$-module homomorphisms are  inclusions $A_v \subset A_w$ (composed with scalar multiplication by an element of $\k$).
\end{enumerate}
\end{prop}

\begin{rem} In addition,  conic modules are always Cohen-Macaulay, hence reflexive. This is proved in  Corollary \ref{CM}.\end{rem}

\begin{rem}
It is tempting to suspect that $\Hom_{R}(A_v, A_w)$ might be the conic module $A_{w-v}$, or perhaps another conic module. This is false in general; in Example \ref{ConeOverSquare}, $\Hom_R(A_1,A_2)$ is not even Cohen-Macaulay. However, we show in Proposition \ref{simplicial} that $\Hom_{R}(A_v, A_w)$ is always conic in the special case that $R$ is simplicial. 
\end{rem}

\begin{proof}
For (1), note that  $A_v\subset \k[M]$ by definition, so   it is clearly a torsion-free $R$-module. Furthermore,  tensoring over $R$ with the Laurent ring $\k[M]$,  we have 
 $A_v\otimes_R \k[M] \cong \k[M],$ so $A_v$ is rank one. 
 
For (2),  observe that the conic module $A_v$ is the span of $x^m \in \k[M]$ where $m$ is in the set
$$
\begin{aligned}
M \cap (C+v)  &= 
\{m \in M \, | \, m - v \in C\} \\ &= 
\{m \in M \, | \, \langle  m - v , n_i \rangle \geq 0\,\,\,\, {\text{for all }} n_i \in \Sigma_1 \} \\ &=
 \{m \in M \, | \, \langle  m, n_i \rangle \geq   \langle  v, n_i \rangle  \,\,\,\, {\text{for all }} n_i \in \Sigma_1 \}\\ & = 
  \{m \in M \, | \, \langle  m, n_i \rangle \geq   \lceil \langle  v, n_i \rangle \rceil \,\,\,\, {\text{for all }} n_i \in \Sigma_1 \}\\
  &= M \cap \left(\bigcap_{n_i \in \Sigma_1} \{x\in M_{\mathbb R} \, | \,  \langle  x, n_i \rangle \geq   \lceil \langle  v, n_i \rangle \rceil \}\right) \ .
 \end{aligned}
  $$
  The penultimate equality here follows because $  \langle  m, n_i \rangle$ is an integer.
  
For (3), note that since $A_v$ and $A_w$ are finitely generated and $M$-graded, so is 
$\Hom_{R}(A_v, A_w)$. Tensoring with  the Laurent ring $\k[M]$, we see that $\Hom_R(A_v, A_w)$ is an $R$-submodule of $   \Hom_{\k[M]}(\k[M], \k[M]) \cong \k[M]$ so that every 
homogeneous $R$-module homomorphism $A_v\rightarrow A_w$ is multiplication by some monomial in $\k[M]$.
The statement follows. In particular,  the only degree-preserving maps are $\mu_{\lambda}: A_v\rightarrow A_w$ sending each $x^m$ to $\lambda x^m$, where $\lambda \in \k$. 
\end{proof}

\subsection{Open Conic Modules}
For any $v\in M_{\mathbb R}$, we define the associated  {\bf open conic module} to be the $M$-graded  $R$-submodule of $\k[M]$:
$$\mathring{A}_{v}
:= \Span\, \{x^m \, | \, m \in M \cap  (\mathring{C}+v)  \},
$$ where 
$ \mathring{C}+v$ denotes the {\it interior} of the cone $C+v.$

Open conic modules turn out to be conic modules of ``nearby" vectors in $M_{\mathbb R}$:

\begin{prop}\label{OpenConic1}
Fix a toric algebra $R$ corresponding to a full-dimensional cone $C$ in $M_{\mathbb R}$.  For every  open conic $R$ module $\mathring{A}_{v}$, there exists  $w \in M_{\mathbb R}$ such that 
$$\mathring{A}_{v} = A_w.$$
Explicitly, we can take $w = v+\varepsilon p$ where  $p$ is any vector in the {\it  interior}  of $C$ and $\varepsilon$ is a sufficiently small positive real number. 
\end{prop}

\begin{proof}[Proof of Proposition \ref{OpenConic1}]
Fix any 
 $p \in M_{\mathbb R}$ in the interior of $C$. This means that $\langle p, n_i\rangle > 0$ for all $n_i \in \Sigma_1$. So for any $m\in M$, we have
$$
\langle m-v, n_i\rangle > 0 \,\,\,\,\,{\text{if and only if }} \,\,\,\, \langle m  -(v+ \varepsilon p), n_i\rangle \geq 0
$$
for sufficiently small positive $\varepsilon$. 
This exactly says that $\mathring{A_{v}} = A_{v+\varepsilon p}$.
\end{proof}

\section{Chambers of Constancy and Cell Decomposition}  \label{Notation3} 

Through this section, $C$ denotes  a full-dimensional rational polyhedral cone $C$ in $M_{\mathbb R} = M\otimes \RR$.  As usual,  $R = \k[C\cap M]$ denotes  the corresponding toric  ring and 
$\Sigma_1$ the set of primitive inward-pointing normals to its facets. 

Different vectors $v \in M_{\mathbb R}$ can determine the same conic module over  $R$. 
This leads to the following partition of $M_{\mathbb R}$ into equivalence classes of vectors determining the same conic module:
\begin{defn} 
The {\bf chamber of constancy } $\Delta \subset M_{\mathbb R}$  containing $v \in M_{\mathbb R}$ is the set of 
all $w\in M_{\mathbb R}$ such that $A_v = A_w$. Equivalently, $w$ and $v$ belong to the same chamber of constancy if and only if $(C+v) \cap M = (C+w)\cap M$.
\end{defn}

\noindent
\begin{notation} 
Because two vectors determine the same conic module if and only if they belong to the same chamber $\Delta$, it is natural to denote a conic module $A_{\Delta}$. 
\end{notation}

\begin{prop}\label{Equations} 
The  chamber of constancy containing $v$ is the set 
\begin{equation}\label{ChamberEquations}
\begin{aligned}
\Delta & = \{x \in M_{\mathbb R}\,\, | \,\, \lceil \langle v, n_i \rangle \rceil  - 1 <    \langle x, n_i \rangle \leq  \lceil \langle v, n_i   \rangle\rceil  \,\,\,\, {\text{for all }} \,\, n_i \in \Sigma_1\}\\
& =  \bigcap_{n_i \in \Sigma_1} \{x \in M_{\mathbb R}\,\, | \,\, \lceil \langle v, n_i \rangle \rceil  - 1 <   \langle x, n_i \rangle  \leq  \lceil \langle v, n_i   \rangle\rceil \}
\end{aligned}
\end{equation}
 where  
 $n_1, \dots, n_t$ runs through the set $\Sigma_1$  of primitive inward-pointing normal vectors to the facets of $C$.
 \end{prop}
 
Before proving this, we note the following immediate corollary. 
 
 \begin{coro}\label{polyhdral}  Fix a  full-dimensional rational polyhedral cone $C \subset M_{\mathbb R}$. The corresponding chambers of constancy  decompose
 $ M_{\mathbb R}$ into disjoint
 locally closed polyhedral regions;  these chambers are  bounded if $C$ is pointed. \end{coro}

 \begin{proof}[Proof of Corollary]
 Proposition  \ref{Equations} expresses a chamber $\Delta$ as an intersection of finitely many half-open strips formed by the intersection of a  closed half space and a parallel open half space. Thus it is locally closed and polyhedral. If $C$ is pointed, then its dual $\sigma$ is full-dimensional, so $\Sigma_1$ spans $N_{\mathbb R}$, making $\Delta$ bounded. 
 \end{proof}

 \begin{proof}[Proof of Proposition  \ref{Equations}]
Fix $v \in M_{\mathbb R}$.  The Lemma is equivalent to the claim that $A_w = A_v$ if and only if the two lists of integers
$$ \lceil \langle  v, n_1 \rangle \rceil ,  \lceil \langle  v, n_2 \rangle \rceil, \dots,  \lceil \langle  v, n_t \rangle \rceil
$$ and
$$ \lceil \langle  w, n_1 \rangle \rceil ,  \lceil \langle  w, n_2 \rangle \rceil, \dots,  \lceil \langle  w, n_t \rangle \rceil
$$
are the same. Clearly, if $\lceil \langle  v, n_i \rangle \rceil =   \lceil \langle  w, n_i \rangle \rceil$ for all $n_i \in \Sigma_1$, then Proposition
 \ref{ConicModProp}  (2) ensures that  $A_v = A_w$.
It remains to prove the converse.

Let us assume, without loss of generality,  that $\lceil \langle  v, n_1 \rangle \rceil <   \lceil \langle  w, n_1 \rangle \rceil$. We need to show that $A_v\neq A_w$.
For this, we will find a monomial $x^m\in A_v$ but not in $A_w$.

Pick any $m \in M \cap  \mathcal H$, where   $\mathcal H$ is the hyperplane in $M_{\mathbb R}$ 
 whose equation is $\langle  x, n_1\rangle = \lceil\langle  v, n_1\rangle \rceil$. Note that  $m$ is {\it not} in the closed half space
  $$
 \{x\in  M_{\mathbb R} \, |  \langle  x, n_1\rangle \geq  \lceil\langle  w, n_1\rangle \rceil \},
  $$ so that $m \notin C+w$ by Proposition 
   \ref{ConicModProp}  (2).   In particular, $x^m \notin A_w$ for any $m \in M \cap  \mathcal H.$ 
   On the other hand,  again by Proposition  \ref{ConicModProp},   the monomials in $A_v$ are those $m\in M$ such that 
   $$
  m\in  \bigcap_{i=1}^t \{x\in M_{\mathbb R} \, | \,  \langle  x, n_i \rangle \geq   \lceil \langle  v, n_i \rangle \rceil \},   $$ 
  a convex subset of $M_{\mathbb R}$ which has non-empty intersection with rational hyperplane $ \mathcal H$ and hence with the lattice $M\cap \mathcal H$.
     Therefore,  we can find  $m\in (C+v)\cap  M \cap \mathcal H$. This $m$ satisfies
   $x^m \in A_v\setminus A_w$.
\end{proof}

 \begin{rem}\label{ConeList} 
 The proof of Proposition \ref{Equations} implies that a  conic module $A_v$ is uniquely described by the list of integers
 \begin{equation}\label{list}
(\lceil  \langle v, n_1 \rangle \rceil, \lceil  \langle v, n_2 \rangle \rceil, \dots, \lceil  \langle v, n_t \rangle \rceil)  \, \in \mathbb Z^{\Sigma_1}
 \end{equation}
 where the $n_i\in \Sigma_1$ range through the primitive linear forms defining the facets of $C$. 
 This can be understood in terms of the group of torus-invariant Weil divisors on the toric variety $X_{\sigma} = \Spec R$, a  free group generated by the irreducible codimension one toric subvarieties  determined by each of of the facets of $C$, hence is also isomorphic to $ \mathbb Z^{\Sigma_1}$.
  The conic module $A_v$  is the module of global section of the reflexive sheaf 
  $\mathcal O_{X_{\sigma}}(-D)$, where $D$ is the divisor
  $$\lceil\langle v, n_1\rangle\rceil D_1+ \lceil\langle v, n_2\rangle\rceil D_2 +  \dots + \lceil\langle v, n_t\rangle\rceil D_{t} \ ,$$ 
where $D_i$ is the divisor of the facet indexed by $n_i \in \Sigma_1$.
The reader is cautioned, however, that the conic modules are a {\it strict} subset of the set of all reflexive $M$-graded submodules
  of the fraction field of $\k[M]$ in general, so not every list of integers corresponds to a conic module.{\footnote{In general not even all rank 1 maximal Cohen-Macaulay modules are conic, see \cite{Baetica}.}}
\end{rem}

\subsection{CW Decomposition of $M_{\mathbb R}$.}

The cone $C$ in $M_{\mathbb R}$ determines a CW decomposition of $M_{\mathbb R}$ whose maximal dimensional cells are the interiors of the chambers of constancy. More precisely: 

\begin{prop}\label{CellDecomp}  \cite{Bruns/arXiv:math/0408110} Fix a full dimensional  rational polyhedral cone $C$ in $M_{\mathbb R}.$ 
Each of its chambers of constancy in $M_{\mathbb R}$  is  a disjoint union of open polyhedral  cells (of different dimensions).  Together,   ranging over all chambers, these cells define a CW decomposition of $M_{\mathbb R}$.
 \end{prop}

 \begin{rem}
 We caution the reader that a chamber of constancy $\Delta$ need not have cells of every dimension $\leq \dim M_{\mathbb R}$. See Section \ref{Sec7}.
 \end{rem}
 
\begin{proof}[Proof of Proposition \ref{CellDecomp}]
Fix a chamber of constancy $\Delta$ containing a vector $v \in M_{\mathbb R}$. We will explicitly describe $\Delta$ as a 
finite union of open polyhedral cells. 

There is exactly one open cell of codimension zero, namely the interior of $\Delta$.
This is the set of vectors in $v \in M_{\mathbb R}$ satisfying 
\begin{equation}\label{inequalities}
 \lceil \langle v, n_i \rangle \rceil - 1  <  \langle x, n_i \rangle <  \lceil  \langle v, n_i \rangle \rceil
 \end{equation} 
for all $n_i \in \Sigma_1$, the set of primitive inward-pointing normals to the facets of $C$.
We can think of these inequalities as the complete set of inequalities defining $\Delta$  as in Proposition \ref{Equations} but in which {\it all inequalities are strict.}

The open cells of codimension 1 are the interiors of the {\it facets} of $\Delta$. These cells are defined as  the points of $M_{\mathbb R}$ satisfying the same set of  strict linear  inequalities as above in (\ref{inequalities}), but with exactly  one of the linear inequalities replaced by an 
{\it equality} of the form  $\langle x, n_i \rangle = \lceil \langle v, n_i \rangle \rceil$ for some $n_i \in \Sigma_1$. 

The open cells of codimension 2 are similarly defined by the same set of strict  linear inequalities  (\ref{inequalities}) but with  two equalities. These cells are open polytopes in the codimension two linear spaces
 in $M_{\mathbb R}$ obtained by intersecting  two supporting hyperplanes of $\Delta$.  In a similar way, each cell of codimension $i$ of $\Delta$ is an open polytope in a codimension $i$ linear space of $M_{\mathbb R}$ obtained by intersecting $i$ supporting hyperplanes of  $\Delta$.
 
 Summarizing formally, we can describe each  open cell of $\Delta$ as follows. Pick any $v\in \Delta$ and set $d_j = \lceil \langle v, n_j \rangle \rceil$ for each $n_j \in \Sigma_1$. Then 
\begin{equation}
\label{CellFormula}
 \tau_{\Omega}  :=  \left\{x\in M_{\mathbb R}\, \,\,\, \Large{|} \, \, \,\,\
	\begin{array}{ll}
		 d_j - 1 < \langle x, n_j \rangle < d_j  & \mbox{for }  n_j\in \Omega,  \\
		\langle x, n_j \rangle = d_j  & \mbox{for } n_j \not\in \Omega 
	\end{array}
\right\}.
\end{equation}
where $\Omega$ is some subset of the set $\Sigma_1$ of primitive inward-pointing normals of $C$.

The space $M_{\mathbb R}$ is clearly the disjoint union of these polyhedral cells, as we range over all cells of all  chambers. 
It is straightforward to check that this cell decomposition satisfies the definition of a CW complex. (See, for example, \cite[\S 38]{Munkres}.)	
\end{proof}

\begin{rem}\label{CompareClosed}
The cell decomposition of $\Delta$  is not a CW complex because some boundary cells are missing. 
However, 
the closure $\overline{\Delta} $ of $\Delta$ is a CW subcomplex of  $M_{\mathbb R}$  in an obvious  way.  \end{rem}

The  CW decomposition  of $M_\mathbb{R}$ can be clarified using open conic modules:
\begin{coro}\label{OpenCells} With notation as above, 
two vectors  $v,w\in M_{\mathbb{R}}$ lie in the same open cell of the CW decomposition  of $M_\mathbb{R}$  if and only if 
\[ A_v  = A_w \,\,\,\text{ and }\, \,\,\,\mathring{A}_v = \mathring{A}_w.\]

Moreover,  for $v$ and $w$ lying in the same chamber of constancy $\Delta$ we have 
$$\mathring{A_{v}} \subset \mathring{A_{w}}$$
if and only if the open cell  containing $v$ is contained in the boundary of the open cell  containing $w$.
 \end{coro}

\begin{proof}  Fix  $v, w \in \Delta$. It suffices to show that $v$ and $w$ are in the 
same open cell of $\Delta$ if and  only if $\mathring{A}_v = \mathring{A}_w.$

Fix $p$ in the interior of $\Delta$. According to Proposition \ref{OpenConic1},
$\mathring{A}_v = \mathring{A}_w$  if and only only if
\begin{equation}\label{eq5}
\lceil \langle v+\varepsilon p, n_j \rangle \rceil = \lceil \langle w+\varepsilon p, n_j \rangle \rceil 
\end{equation}
for all $n_j \in \Sigma_1$ and sufficiently small $\varepsilon$.
But (\ref{eq5}) holds if and only if 
$$
\langle v, n_j \rangle \in \mathbb Z  \Longleftrightarrow
\langle w, n_j \rangle \in \mathbb Z 
$$
in which case both values are $d_j = \lceil \langle v, n_j \rangle \rceil  = \lceil \langle w, n_j \rangle \rceil $.
But this means exactly that both $v$ and $w$ lie in the open cell
\begin{equation}\label{eq7}
  \left\{x\in M_{\mathbb R}\, \,\,\, \large{|} \, \, \,\,\
	\begin{array}{ll}
		 d_j - 1 < \langle x, n_j \rangle < d_j  & \mbox{for }  j \in \Omega \\
		\langle x, n_j \rangle = d_j  & \mbox{for } j \not\in \Omega
	\end{array}
\right\}.
\end{equation}
where $\Omega$ is the set of indices for which $ \langle v, n_j \rangle$ and $ \langle w, n_j \rangle$
are {\it not}  integers.

For the second statement,  we 
 observe that $\tau$ is in the boundary of $\tau'$ precisely when $\tau$ is the intersection of $ \tau'$ with some supporting hyperplanes for $\Delta$ defined by  linear equations $\langle x, n_k \rangle = \lceil \langle v, n_k\rangle \rceil.$ So clearly if $m\in M$ satisfies all the defining (in)equalities  to be in the interior of $C+v$ (where $v\in \tau$), then it also satisfies the less restrictive conditions to be in the interior of  $C+w$ (where $w\in \tau'$). This says exactly that $  \mathring A_v \subset   \mathring A_w$.
\end{proof}

\noindent
\begin{notation} In  light of Corollary \ref{OpenCells}, it is reasonable to denote an open conic module   $\mathring A_{v}$ by  $\mathring A_{\tau}$ 
where $\tau$ is the open cell containing $v$. However, it can happen that  $\mathring A_{\tau} =  \mathring A_{\tau'}$  for distinct cells $\tau$ and $\tau'$, so we will  
use this notation only when considering open conic modules  $\mathring A_{v}$ for $v$ in a fixed chamber of constancy $\Delta$. 
\end{notation}

\subsection{Isomorphism Classes of Conic Modules.} Isomorphic conic modules can be understood in terms of the chamber decomposition of $M_{\mathbb R}$: 

\begin{prop}\label{isom}
Two conic $R$-modules $A_{\Delta}$ and $A_{\Delta'}$ are isomorphic if and only if there is an $m\in M$ such that $\Delta = \Delta' +m$.
\end{prop}

\begin{proof}
Because conic modules are $M$-graded rank one torsion-free modules over the normal ring $R$, any $M$-graded isomorphism between them is given by multiplication by some monomial $x^m$ of the Laurent ring $\k[x^M]$. So $A_v \cong A_w$ if and only if $A_{w} = A_{v+m}$ for some $m\in M$. \end{proof}

Proposition \ref{isom} says that  $M$ acts naturally by translation  on the set of chambers of constancy. 
The orbits of this action index the distinct
 isomorphism classes of conic modules just as the chambers themselves index  distinct conic modules.
   Put differently, there is an induced CW decomposition of the real torus $M_{\mathbb R}/M$ whose maximal dimensional cells are in one-one-correspondence with the isomorphism classes of conic modules for $R$. The compactness of $M_{\mathbb R}/M$ has the  following immediate consequence:  
 \begin{coro} \cite[Prop.~3.6]{BrunsGubeladze} 
For any toric algebra $R$, there are finitely many isomorphism types of conic $R$-modules.\footnote{{For a different interpretation cf.~\cite{Brion, VdB-CM-tori} where a criterion for Cohen-Macaulayness of modules of covariants is given, and \cite[Prop.~3.2.3]{SmVdB} that shows that $R^{1/q}$ is a direct sum of modules of covariants belonging to so-called strongly critical characters.}}
\end{coro}

\begin{ex}The cone $C$ in Figure \ref{fig: coneA} induces the CW decomposition of the plane in Figure \ref{fig: cellsA}. The  chambers of constancy
 in $M_{\mathbb{R}} = \mathbb R^2$ are the connected unions of same-colored cells.

\begin{figure}[h!t]
\centering
\begin{subfigure}[b]{.45\textwidth}
\centering
\begin{tikzpicture}[baseline=(current bounding box.center),scale=1.25]
	\clip (-2.25,-1.25) rectangle (2.25,2.25);
	\draw[fill=black!10] (-4,4) to (0,0) to (4,4);
	\node[dot] at (-2,2) {};
	\node[dot] at (-1,2) {};
	\node[dot] at (0,2) {};
	\node[dot] at (1,2) {};
	\node[dot] at (2,2) {};
	\node[faint dot] at (-2,1) {};
	\node[dot] at (-1,1) {};
	\node[dot] at (0,1) {};
	\node[dot] at (1,1) {};
	\node[faint dot] at (2,1) {};
	\node[faint dot] at (-2,0) {};
	\node[faint dot] at (-1,0) {};
	\node[dot] at (0,0) {};
	\node[faint dot] at (1,0) {};
	\node[faint dot] at (2,0) {};
	\node[faint dot] at (-2,-1) {};
	\node[faint dot] at (-1,-1) {};
	\node[faint dot] at (0,-1) {};
	\node[faint dot] at (1,-1) {};
	\node[faint dot] at (2,-1) {};
\end{tikzpicture}
\caption{The cone $C$.}
\label{fig: coneA}
\end{subfigure}
\begin{subfigure}[b]{.45\textwidth}
\centering
\begin{tikzpicture}[baseline=(current bounding box.center),scale=1.25]
	\clip (-2.25,-1.25) rectangle (2.25,2.25);
	\begin{scope}[yshift=2cm]
	\path[fill=blue!10] (-2.5,.5) to (-2,1) to (-1.5,.5) to (-1,1) to (-.5,.5) to (0,1) to (.5,.5) to (1,1) to (1.5,.5) to (2,1) to (2.5,.5) to (2,0) to (1.5,.5) to (1,0) to (.5,.5) to (0,0) to (-.5,.5) to (-1,0) to (-1.5,.5) to (-2,0) to (-2.5,.5);	
	\path[fill=red!10] (-2.5,.5) to (-2,0) to (-1.5,.5) to (-1,0) to (-.5,.5) to (0,0) to (.5,.5) to (1,0) to (1.5,.5) to (2,0) to (2.5,.5) to (2.5,-.5) to (2,0) to (1.5,.-.5) to (1,0) to (.5,-.5) to (0,0) to (-.5,-.5) to (-1,0) to (-1.5,-.5) to (-2,0) to (-2.5,-.5);	
	\draw[blue!75!black] (-2.5,.5) to (-2,1) to (-1.5,.5) to (-1,1) to (-.5,.5) to (0,1) to (.5,.5) to (1,1) to (1.5,.5) to (2,1) to (2.5,.5);
	\draw[red!75!black] (-2.5,.5) to (-2,0) to (-1.5,.5) to (-1,0) to (-.5,.5) to (0,0) to (.5,.5) to (1,0) to (1.5,.5) to (2,0) to (2.5,.5);
	\end{scope}
	\begin{scope}[yshift=1cm]
	\path[fill=blue!10] (-2.5,.5) to (-2,1) to (-1.5,.5) to (-1,1) to (-.5,.5) to (0,1) to (.5,.5) to (1,1) to (1.5,.5) to (2,1) to (2.5,.5) to (2,0) to (1.5,.5) to (1,0) to (.5,.5) to (0,0) to (-.5,.5) to (-1,0) to (-1.5,.5) to (-2,0) to (-2.5,.5);	
	\path[fill=red!10] (-2.5,.5) to (-2,0) to (-1.5,.5) to (-1,0) to (-.5,.5) to (0,0) to (.5,.5) to (1,0) to (1.5,.5) to (2,0) to (2.5,.5) to (2.5,-.5) to (2,0) to (1.5,.-.5) to (1,0) to (.5,-.5) to (0,0) to (-.5,-.5) to (-1,0) to (-1.5,-.5) to (-2,0) to (-2.5,-.5);	
	\draw[blue!75!black] (-2.5,.5) to (-2,1) to (-1.5,.5) to (-1,1) to (-.5,.5) to (0,1) to (.5,.5) to (1,1) to (1.5,.5) to (2,1) to (2.5,.5);
	\draw[red!75!black] (-2.5,.5) to (-2,0) to (-1.5,.5) to (-1,0) to (-.5,.5) to (0,0) to (.5,.5) to (1,0) to (1.5,.5) to (2,0) to (2.5,.5);
	\end{scope}
	\begin{scope}[yshift=0cm]
	\path[fill=blue!10] (-2.5,.5) to (-2,1) to (-1.5,.5) to (-1,1) to (-.5,.5) to (0,1) to (.5,.5) to (1,1) to (1.5,.5) to (2,1) to (2.5,.5) to (2,0) to (1.5,.5) to (1,0) to (.5,.5) to (0,0) to (-.5,.5) to (-1,0) to (-1.5,.5) to (-2,0) to (-2.5,.5);	
	\path[fill=red!10] (-2.5,.5) to (-2,0) to (-1.5,.5) to (-1,0) to (-.5,.5) to (0,0) to (.5,.5) to (1,0) to (1.5,.5) to (2,0) to (2.5,.5) to (2.5,-.5) to (2,0) to (1.5,.-.5) to (1,0) to (.5,-.5) to (0,0) to (-.5,-.5) to (-1,0) to (-1.5,-.5) to (-2,0) to (-2.5,-.5);	
	\draw[blue!75!black] (-2.5,.5) to (-2,1) to (-1.5,.5) to (-1,1) to (-.5,.5) to (0,1) to (.5,.5) to (1,1) to (1.5,.5) to (2,1) to (2.5,.5);
	\draw[red!75!black] (-2.5,.5) to (-2,0) to (-1.5,.5) to (-1,0) to (-.5,.5) to (0,0) to (.5,.5) to (1,0) to (1.5,.5) to (2,0) to (2.5,.5);
	\end{scope}
	\begin{scope}[yshift=-1cm]
	\path[fill=blue!10] (-2.5,.5) to (-2,1) to (-1.5,.5) to (-1,1) to (-.5,.5) to (0,1) to (.5,.5) to (1,1) to (1.5,.5) to (2,1) to (2.5,.5) to (2,0) to (1.5,.5) to (1,0) to (.5,.5) to (0,0) to (-.5,.5) to (-1,0) to (-1.5,.5) to (-2,0) to (-2.5,.5);	
	\path[fill=red!10] (-2.5,.5) to (-2,0) to (-1.5,.5) to (-1,0) to (-.5,.5) to (0,0) to (.5,.5) to (1,0) to (1.5,.5) to (2,0) to (2.5,.5) to (2.5,-.5) to (2,0) to (1.5,.-.5) to (1,0) to (.5,-.5) to (0,0) to (-.5,-.5) to (-1,0) to (-1.5,-.5) to (-2,0) to (-2.5,-.5);	
	\draw[blue!75!black] (-2.5,.5) to (-2,1) to (-1.5,.5) to (-1,1) to (-.5,.5) to (0,1) to (.5,.5) to (1,1) to (1.5,.5) to (2,1) to (2.5,.5);
	\draw[red!75!black] (-2.5,.5) to (-2,0) to (-1.5,.5) to (-1,0) to (-.5,.5) to (0,0) to (.5,.5) to (1,0) to (1.5,.5) to (2,0) to (2.5,.5);
	\end{scope}
	\begin{scope}[yshift=-2cm]
	\path[fill=blue!10] (-2.5,.5) to (-2,1) to (-1.5,.5) to (-1,1) to (-.5,.5) to (0,1) to (.5,.5) to (1,1) to (1.5,.5) to (2,1) to (2.5,.5) to (2,0) to (1.5,.5) to (1,0) to (.5,.5) to (0,0) to (-.5,.5) to (-1,0) to (-1.5,.5) to (-2,0) to (-2.5,.5);	
	\path[fill=red!10] (-2.5,.5) to (-2,0) to (-1.5,.5) to (-1,0) to (-.5,.5) to (0,0) to (.5,.5) to (1,0) to (1.5,.5) to (2,0) to (2.5,.5) to (2.5,-.5) to (2,0) to (1.5,.-.5) to (1,0) to (.5,-.5) to (0,0) to (-.5,-.5) to (-1,0) to (-1.5,-.5) to (-2,0) to (-2.5,-.5);	
	\draw[blue!75!black] (-2.5,.5) to (-2,1) to (-1.5,.5) to (-1,1) to (-.5,.5) to (0,1) to (.5,.5) to (1,1) to (1.5,.5) to (2,1) to (2.5,.5);
	\draw[red!75!black] (-2.5,.5) to (-2,0) to (-1.5,.5) to (-1,0) to (-.5,.5) to (0,0) to (.5,.5) to (1,0) to (1.5,.5) to (2,0) to (2.5,.5);
	\end{scope}

	\node[blue dot] at (-2,2) {};
	\node[blue dot] at (-1,2) {};
	\node[blue dot] at (0,2) {};
	\node[blue dot] at (1,2) {};
	\node[blue dot] at (2,2) {};
	\node[red dot] at (-1.5,1.5) {};
	\node[red dot] at (-.5,1.5) {};
	\node[red dot] at (.5,1.5) {};
	\node[red dot] at (1.5,1.5) {};
	\node[blue dot] at (-2,1) {};
	\node[blue dot] at (-1,1) {};
	\node[blue dot] at (0,1) {};
	\node[blue dot] at (1,1) {};
	\node[blue dot ] at (2,1) {};
	\node[red dot] at (-1.5,.5) {};
	\node[red dot] at (-.5,.5) {};
	\node[red dot] at (.5,.5) {};
	\node[red dot] at (1.5,.5) {};
	\node[blue dot ] at (-2,0) {};
	\node[blue dot] at (-1,0) {};
	\node[blue dot] at (0,0) {};
	\node[blue dot] at (1,0) {};
	\node[blue dot] at (2,0) {};
	\node[red dot] at (-1.5,-.5) {};
	\node[red dot] at (-.5,-.5) {};
	\node[red dot] at (.5,-.5) {};
	\node[red dot] at (1.5,-.5) {};
	\node[blue dot] at (-2,-1) {};
	\node[blue dot] at (-1,-1) {};
	\node[blue dot] at (0,-1) {};
	\node[blue dot] at (1,-1) {};
	\node[blue dot] at (2,-1) {};
	
	\node[blue] at (0,-.5) {$R$};
	\node[red] at (.5,0) {$A$};
\end{tikzpicture}
\caption{The CW decomposition of the plane.}
\label{fig: cellsA}
\end{subfigure}
\caption{A cone and the corresponding chambers.}
\end{figure}
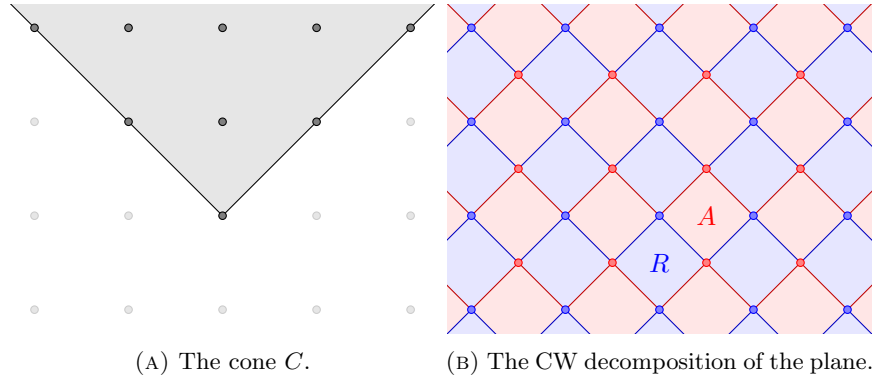

The ring $R = \k[C\cap \mathbb Z^2]$ has presentation
\[ R= \mathsf{k}[x,y,z]/(xz-y^2). \] 
There are two isomorphism classes of conic $R$-modules, represented by $A_0= R$ (the blue chambers) and  
$A = Ry+Rz$ (the red chambers).
\end{ex}

\begin{ex}\label{ConeOverSquare}
Consider the cone $C$  in $\mathbb{R}^3$ over the  unit square depicted in Figure 2. The corresponding toric algebra $R = \k[C\cap \mathbb Z^3]$ has presentation  $ \k[x, y, z, w]/(xz-yw)$. There are three isomorphism types of conic $R$-modules, represented by  $R=A_0$, $(x, y)R = A_1$ and $(x, w)R = A_2$. These are represented (up to translation by $\mathbb Z^3$) by the three chambers depicted in Figure 3. 
Figure 4  depicts several subsets of the cellular neighborhood of a point in $\mathbb R^3$. 
\end{ex}

\begin{figure}[h!t]
\begin{tikzpicture}
	\clip (-2.25,-.5) rectangle (2.25,3);
	\node[coordinate] (000) at \ThreeD{0}{0}{0} {};
	\node[coordinate] (002) at \ThreeD{0}{0}{2} {};
	\node[coordinate] (202) at \ThreeD{2}{0}{2} {};
	\node[coordinate] (022) at \ThreeD{0}{2}{2} {};
	\node[coordinate] (222) at \ThreeD{2}{2}{2} {};

	\draw[fill=black!10] (202) to (000) to (022) to (002) to (202);
	\draw (000) to (002);
	\draw (000) to (222);
	\draw (202) to (222) to (022);

	\node[dot] (000) at \ThreeD{0}{0}{0} {};
	
	\node[dot] (001) at \ThreeD{0}{0}{1} {};
	\node[dot] (101) at \ThreeD{1}{0}{1} {};
	\node[dot] (011) at \ThreeD{0}{1}{1} {};
	\node[dot] (111) at \ThreeD{1}{1}{1} {};
	
	\node[dot] (002) at \ThreeD{0}{0}{2} {};
	\node[dot] (102) at \ThreeD{1}{0}{2} {};
	\node[dot] (202) at \ThreeD{2}{0}{2} {};
	\node[dot] (012) at \ThreeD{0}{1}{2} {};
	\node[dot] (112) at \ThreeD{1}{1}{2} {};
	\node[dot] (212) at \ThreeD{2}{1}{2} {};
	\node[dot] (022) at \ThreeD{0}{2}{2} {};
	\node[dot] (122) at \ThreeD{1}{2}{2} {};
	\node[dot] (222) at \ThreeD{2}{2}{2} {};
	\draw[fill=black!25] (001) to (101) to (111) to (011) to (001);
\end{tikzpicture}
\caption{The cone $C$ over a square.}
\end{figure}
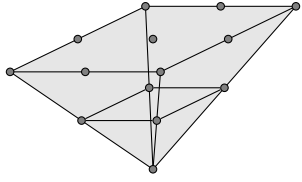

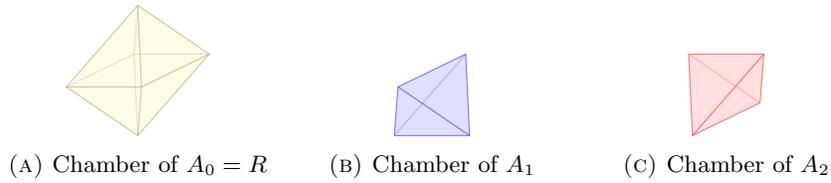
\begin{figure}[h!t]
\centering
\begin{subfigure}[b]{.3\textwidth}
\centering
\begin{tikzpicture}[baseline=(current bounding box.center)]
	\ychamber
\end{tikzpicture}
\caption{Chamber of $A_0 = R$}
\end{subfigure}
\begin{subfigure}[b]{.3\textwidth}
\centering
\begin{tikzpicture}[baseline=(current bounding box.center)]
	\bchamber
\end{tikzpicture}
\caption{Chamber of $A_1$}
\end{subfigure}
\begin{subfigure}[b]{.3\textwidth}
\centering
\begin{tikzpicture}[baseline=(current bounding box.center)]
	\rchamber
\end{tikzpicture}
\caption{Chamber of  $A_2$}
\end{subfigure}
\caption{The three types of Chambers of Constancy.}
\label{fig: chambersC}
\end{figure}

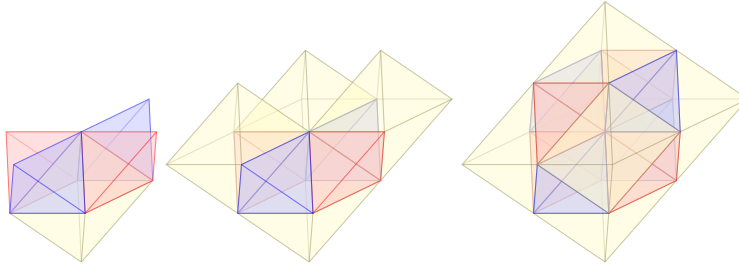
\begin{figure}[h!t]
\begin{tikzpicture}

\begin{scope}[xshift=2cm]
	\rchamber[\ThreeD{-1}{0}{0}]
	\bchamber[\ThreeD{0}{-1}{0}]
	\ychamber[\ThreeD{0}{0}{0}]
	\rchamber[\ThreeD{0}{0}{0}]
	\bchamber[\ThreeD{0}{0}{0}]
\end{scope}
\end{tikzpicture}
\begin{tikzpicture}
	\ychamber[\ThreeD{0}{0}{1}]
	\rchamber[\ThreeD{-1}{0}{0}]
	\bchamber[\ThreeD{0}{-1}{0}]
	\ychamber[\ThreeD{0}{0}{0}]
	\ychamber[\ThreeD{1}{0}{1}]
	\ychamber[\ThreeD{0}{1}{1}]
	\rchamber[\ThreeD{0}{0}{0}]
	\bchamber[\ThreeD{0}{0}{0}]

\end{tikzpicture}
\begin{tikzpicture}
	\ychamber[\ThreeD{0}{0}{1}]
	\rchamber[\ThreeD{0}{0}{1}]
	\bchamber[\ThreeD{0}{0}{1}]
	\rchamber[\ThreeD{-1}{0}{0}]
	\bchamber[\ThreeD{0}{-1}{0}]
	\ychamber[\ThreeD{0}{0}{0}]
	\ychamber[\ThreeD{1}{0}{1}]
	\ychamber[\ThreeD{0}{1}{1}]
	\ychamber[\ThreeD{1}{1}{2}]
	\rchamber[\ThreeD{0}{0}{0}]
	\bchamber[\ThreeD{0}{0}{0}]
	\rchamber[\ThreeD{0}{1}{1}]
	\bchamber[\ThreeD{1}{0}{1}]
	\ychamber[\ThreeD{1}{1}{1}]
\end{tikzpicture}
\caption{Unions of chambers}
\label{fig: keplerstar}
\end{figure}

\subsection{Frobenius and Cohen-Macaulayness}
Fix a positive integer $q$.  Let  $q^{-1}\mathbb Z$ be the cyclic subgroup of $\mathbb Q$ generated by $\frac{1}{q}$ and let
  $q^{-1}M$ be the abelian group $ q^{-1}\mathbb Z \otimes M$. Note that 
  $q^{-1}M$ is a lattice containing  $M$ as an index $q^{d}$ subgroup, where $d $ is the rank of $M$, and that $q^{-1}M_{\mathbb R} = M_{\mathbb R}$.

Consider the toric algebra  
$$
R^{1/q}:= \k[C\cap q^{-1}M] = \Span\{x^{m/q} \, \, | \, \, \frac{m}{q} \in C\cap q^{-1}M\}.
$$
This notation is meant to be suggestive: if $q = p^e$ where the ground field $\k$ is a perfect of characteristic $p$, then $R^{1/q}$ is the ring of $p^e$-th roots of all elements in $R = \k[C\cap M]$. However, the toric algebra $R^{1/q}$  makes sense in general for all natural numbers $q$,  even in characteristic zero.

The toric algebra $R^{1/q}$  is naturally graded by the lattice  $q^{-1}M$.  It contains the sub-algebra $R = \k[C\cap M]$ as the 	subring of elements whose degrees are in the sublattice $M$.

\begin{prop}\label{ISOM}  \cite{Bruns/arXiv:math/0408110}
The $R$-module $R^{1/q}$ decomposes as a direct sum of  (finitely many) modules isomorphic to conic $R$-modules. For sufficiently large $q$, every conic module appears (up to isomorphism) as a summand of $R^{1/q}$. 
\end{prop}

\begin{proof} 
Considering $R^{1/q}$ as a module over the subring $R = \k[C\cap M]$, 
there is  a  natural grading by the quotient group $q^{-1}M/M.$ The graded pieces give an $R$-module decomposition 
$$
R^{1/q} = \bigoplus_{[\frac{v}{q}]\in q^{-1}M/M} [R^{1/q}]_{[\frac{v}{q}]} = \bigoplus_{[\frac{v}{q}]\in q^{-1}M/M} {\text{Span}} \{x^{\frac{v}{q} + m} \, \, | \, \, \frac{v}{q}+m \in (\frac{v}{q}+M)\cap C\}$$
indexed by the classes $[\frac{v}{q}] =  \frac{v}{q} + M$ of the quotient group $q^{-1}M/M$.  Note that $R^{1/q}$ is finite over its degree $0$-graded piece, $R = \k[C\cap M]$.

The $q^{-1}M/M$-graded pieces of this decomposition are all (isomorphic to)  conic modules. Indeed, fixing any representative $\frac{v}{q}$ for a coset in $q^{-1}M/M,$
the natural map 
$$A_{-\frac{v}{q}} \rightarrow [R^{1/q}]_{[\frac{v}{q}]}$$
given by multiplication by $x^{v/q}$ is easily seen to be an isomorphism. 
 
 To see that every conic module appears as a summand of $R^{1/q}$, we need only observe that we can refine the lattice by a sufficiently large $q$ so that each chamber of constancy $\Delta$ contains a lattice point in $q^{-1}M$. 
  \end{proof}

 \begin{coro}\label{CM} The conic modules of a toric algebra are all  Cohen-Macaulay. 
\end{coro}

\begin{proof}
Toric algebras are Cohen-Macaulay. Thus $R^{1/q}$ is a Cohen-Macaulay ring. Being finite over the subring $R$, $R^{1/q}$ is Cohen-Macaulay as an $R$-module.  Thus each of the $R$-module direct summands $R_{[\frac{v}{q}]}$ is a Cohen Macaulay $R$-module.   For sufficiently large $q$, these include all the isomorphism types of conic modules. 
\end{proof}

\subsection{The Poset of Chambers of Constancy}
The set of conic $R$-modules has a natural partial ordering given by inclusion.  Since the conic modules are naturally indexed by the chambers of constancy, we have a corresponding partial ordering of the set of chambers: $\Delta \preceq \Delta'$ if and only if  $A_{\Delta}  \subseteq A_{\Delta'}$. 

Proposition \ref{ConicModProp}(3) can be rephrased using this language as follows:

\begin{prop}\label{Rad}
Let $\Delta$ and $\Delta'$ be chambers of constancy for the toric ring $R$. Then
$$
\Hom_{R}(A_{\Delta}, A_{\Delta'}) = \Span \{x^m \, | \, m+\Delta \preceq \Delta'\}.
$$ 
\end{prop}

\begin{rem}\label{Order}
The partial ordering on chambers  can be described in terms  of the list of integers
$$
 (\lceil  \langle v, n_1 \rangle \rceil, \lceil  \langle v, n_2 \rangle \rceil, \dots, \lceil  \langle v, n_t \rangle \rceil)
$$
determining $\Delta$ as in 
Remark 
\ref{ConeList}. 
 That is, $\Delta \preceq \Delta'$ if and only if their corresponding integer lists $(d_1, d_2, \dots, d_t) $ and $(d_1', d_2', \dots, d_t')$ satisfy 
$-d_ i \leq -d_i' $ for every $i$.  That is,  $\Delta \preceq \Delta'$ if and only if $ \lceil  \langle v, n_i \rangle \rceil \geq  \lceil  \langle v', n_i \rangle \rceil$ for all $n_i \in \Sigma_1$ and any (equivalently, every) $v\in \Delta$ and $v' \in \Delta'$.
\end{rem}

We  assign to each chamber of constancy a numerical value which helps us understand its place in the poset:
\begin{defn}\label{Degree}
The  {\bf degree} of the  chamber of constancy  $\Delta$
is 
$$
\deg(\Delta) = - \sum_{n_i\in \Sigma_1} \lceil \langle v, n_i \rangle \rceil,
$$
 where $v$ is any vector in $\Delta$ and  the $n_i$ run through the finite set $\Sigma_1 \subset \Hom_{\mathbb Z}(M, \mathbb Z)$ of primitive inward-pointing  normals to the facets of $C$. 
\end{defn}

  \begin{rem}\label{DivisorDegree} The  degree of a chamber of constancy $\Delta$ is equal to the degree of the corresponding torus-invariant divisor $D$ on the affine toric variety $X_{\sigma} =  \Spec R$ with the property that  $A_{\Delta}$ is the module of global sections of the associated reflexive sheaf $\mathcal O_{X_{\sigma}}(D)$. See Remark \ref{ConeList}. 
   \end{rem}

\bigskip
The   technical  heart of our main argument proceeds by induction on the degree. The following is a crucial step: 
\begin{lemma}\label{Covering}
Given chambers of constancy $\Delta\succeq\Delta'$, the following are equivalent.
\begin{enumerate}
	\item $\deg(\Delta)-\deg(\Delta')=1$.
	\item The chain $\Delta\succ \Delta'$  is saturated in the poset of chambers of constancy; that is, if $\Delta''$ is a chamber of constancy with $\Delta\succeq \Delta'' \succeq \Delta'$, then either $\Delta''=\Delta$ or $\Delta''=\Delta'$.
	\item $\Delta$ and $\Delta'$ are \emph{adjacent}; that is, they are on opposite sides of a unique hyperplane.
\end{enumerate}
\end{lemma}
\begin{proof}
The proof is straightforward when one considers the linear inequalities defining $\Delta$ and $\Delta'$ given by Proposition \ref{Equations}.
In order to have their degrees differ by exactly one, the only possibility is that exactly one of these linear inequalities differ by exactly 1. So we have  exactly one $n_i\in \Sigma_1$ such that 
 $\Delta'$  has an (intersecting) supporting hyperplane  defined by $\langle x, n_i  \rangle = d  $ (where  $d = \lceil  \langle v', n_i \rangle \rceil $ for $v' \in \Delta'$)
  and  $\Delta$ has (intersecting) supporting hyperplane  $\langle x, n_i \rangle = d+1$; 
This means that $\Delta$ and $\Delta'$ are on opposite sides of their shared supporting hyperplane  whose equation
is   $\langle x, n_i \rangle = d$, and no other module $A_{\Delta''}$ can be squeezed between $A_{\Delta} $ and $A_{\Delta'}$. 
\end{proof}

 \subsection{Restriction to a Hyperplane}\label{Hyperplane}

Let $H$ be a supporting hyperplane of our  rational polyhedral cone $C$ in $ M_{\mathbb R}$. 
Then $C \cap H$ is  a polyhedral cone in the vector space $H,$   rational 
with respect to the  lattice $M\cap H.$ The cone $C\cap H$  is full-dimensional and/or pointed   (in $H$) if $C$ is (in $M_{\mathbb R}$).

As usual, we index the facets of $C$ by  the set $\Sigma_1 = \{n_1, \dots, n_t\}$ of primitive inward-pointing normals. 
We can choose an appropriate hyperplane $H$ so that the cone $C\cap H$ is one of these facets,  say with corresponding normal $n_t$.
The facets of the cone $C\cap H$ are the intersections of the facets of $C$ with $H$, and hence 
  indexed by the (restrictions to $H$ of the) linear functionals 
 $\{{n_1}, \dots, {n_{t-1}}\}.$  These are the primitive inward-pointing normals  
to the facets of 
   $C\cap H$.  
   
   Consider the corresponding toric algebra of the cone $C\cap H$ in $H$: 
$$
S = \k[(C\cap H)\cap (M\cap H)] := \Span\{x^m \,\,\, | \,\,\,  m \in (C\cap H) \cap (M\cap H)\}.
$$
    There is a natural surjection of $\k$-algebras
   $$
   R \twoheadrightarrow S \,\,\,\,\,\,\,\, x^m \mapsto \left\{\begin{array}{cc}  x^m  & {\text{if}} \,\,\,\,  m \in H \\  0 &  {\text{if}}  \,\,\,\, m \not\in H, \end{array} \right.
   $$ 
   whose kernel is the height one  monomial ideal of $R$  
   \begin{equation}\label{Ideal}
   P_{n_t} := \Span\{x^m \, \, | \, \, \langle m, n_t \rangle > 0\}.  
   \end{equation}

\begin{prop}\label{Restrict}
Let $H$ be a supporting hyperplane of our full-dimensional rational polyhedral cone $C$ in $M_{\mathbb R}$ and let 
$$ R\twoheadrightarrow S$$
be the induced surjection of toric algebras corresponding to the rational polyhedral cones  $C$ and $C\cap H$, respectively.
Then 
\begin{enumerate}
\item
The toric $S$-modules are in one-to-one correspondence with toric $R$-modules of vectors $v$ lying in the subspace $ H$ of $ M_{\mathbb R}$.
\item \label{item2}The map $\Delta\mapsto \Delta\cap H$ induces a bijection between the chambers of constancy in $M_{\mathbb R}$ (indexing conic $R$-modules)
 that meet $H$ and the chambers of constancy in $H$ (indexing conic $S$-modules).  
\item The bijection in \eqref{item2} respects the partial order and the  cell decomposition of chambers. In particular, the open cells of $H$ are the same as the open cells of $M_{\mathbb R}$ contained in $H$.
\end{enumerate}
\end{prop}

\begin{proof}
  All statements follow immediately from the definitions and by applying
  Propositions
   \ref{ConicModProp},   \ref{OpenConic1},  and  \ref{Equations}, 
  describing chambers and cells in terms of the  elements of $\Sigma_1$. The point is that  
the restrictions to $H$ of the  linear functionals 
 $\{{n_1}, \dots, {n_{t-1}}\}$ in $\Sigma_1$ for $C$ is the set of primitive inward-pointing normals to $C\cap H$. 
\end{proof}

\section{Chain Complexes of Conic Modules}
Throughout this section, $R$ denotes a toric algebra corresponding to a pointed full-dimensional  rational polyhedral cone $C$ in $M_{\mathbb R} = M \otimes \mathbb R$.

Fix  a conic module $A_{\Delta}$ over $R$,  with  the  chamber of constancy $\Delta \subset M_{\mathbb R}$.  Our goal is to write down a complex $K_{\Delta}^{\bullet}$ of $R$-modules, abutting to  $A_{\Delta}$,  and  consisting only of direct sums of conic modules. In \S \ref{Section6}, we will construct from this complex a  projective  resolution of  a simple module over the endomorphism ring $\End_R(\mathbb A)$ where $\mathbb A$ is a sum of conic modules.  
 
 \subsection{The Chain Complex of a Chamber}
 Fix a chamber of constancy $\Delta$. For each open cell of $\Delta$, arbitrarily  fix some orientation. 
 Note that if  $\tau_i$ is a cell of $\Delta$ of codimension $ i \neq 0$, then  $\tau_i$ is contained in the boundary of 
 some  cell $\tau_{i-1}$ of codimension $i-1$ of $\Delta$. 
 Our choice of orientation on $\tau_{i-1}$ induces an orientation on its boundary components, including $\tau_{i}$, 
  which may or may not agree with our choice of orientation for $\tau_i$. 
 
We define the {\bf chain complex of the chamber $\Delta$} as the 
 complex $K_{\Delta}^{\bullet}$ of  $M$-graded $R$-modules
\begin{equation}\label{K}
\cdots 
\longrightarrow 
\bigoplus_{\stackrel{\text{cells }\tau \in \Delta}{\text{codim}(\tau)=2}} \mathring{A}_{\tau}\,\,\,
\longrightarrow 
\bigoplus_{\stackrel{\text{cells }\tau\in \Delta}{\text{codim}(\tau)=1}} \mathring{A}_{\tau} \,\,\,
\longrightarrow 
 \bigoplus_{\stackrel{\text{cells }\tau\in \Delta}{\text{codim}(\tau)=0}} \mathring{A}_{\tau} \,\,\,
  = A_{\Delta} \,\,\,
\end{equation}
where each map in the complex is a signed  sum of signed inclusion maps 
$ \mathring{A}_{\tau_{i+1}} \hookrightarrow \mathring{A}_{\tau_{i}}$  (whenever $\mathring{A}_{\tau_{i+1}} \subset  \mathring{A}_{\tau_{i}}$) and zero maps 
(whenever $\mathring{A}_{\tau_{i+1}} \not\subset \mathring{A}_{\tau_{i}}$). The sign is determined by our chosen orientations: the sign is 1 if our chosen orientation on $ \tau_{i+1}$ agrees with the orientation induced by our chosen  orientation on $\tau_{i}$ by virtue of  its being a boundary component of $\tau_{i}$, and the sign is $-1$ otherwise.

\begin{lemma}
$K_{\Delta}^{\bullet}$ is an $M$-graded complex of $R$-modules.
\end{lemma}

\begin{proof} We need only 
verify that $K_{\Delta}^{\bullet}$ is a complex, that is, that $d^2 = 0$. For this,  we need to follow just one monomial $x^m\in A_{\tau_{i+1}}$ in one summand of an arbitrary term 
$$ 
\bigoplus_{\stackrel{\text{cells }\tau\in \Delta}{\text{codim}(\tau)=i+1}} \mathring{A}_{\tau}$$
 of  $K_{\Delta}^{\bullet}.
$
Under $d$, this  monomial is sent to the list of signed monomials $\pm x^m$ and zeros in $
\bigoplus_{\stackrel{\text{cells }\tau \in \Delta}{\text{codim}(\tau)=i}} \mathring{A}_{\tau}$, according to whether or not  $\tau_{i+1} \subset \partial \tau_i$ (and if so, whether or not the orientations agree).  Applying $d$ again, the reader can check that the signed monomials come in pairs so cancel out to produce $d^2(x^m) = 0$.  Indeed, the calculation  is exactly the same as that to verify that the cellular cochain complex of the associated CW complex of $\overline{\Delta}$  is a complex; to see this note that if  $\tau$ and $\tau'$ are open cells of 
$\overline\Delta$ such that $\tau \subset \partial \tau'$, and 
 if $\tau$ is a cell of $\Delta$, then so is $\tau'$. Thus the differential of the complex  $K^{\bullet}_{\Delta}$ applied to some monomial $x^m$ in some summand $A_{\tau}$ is the same as it would be in the cellular cohomology complex of the closed cell complex $\overline\Delta$. 
See \cite{Munkres}.  \end{proof}

\begin{rem}
 The  chain complex $K_{\Delta}^{\bullet}$ has length  bounded by the vector space dimension of $M_{\mathbb R}$, which  (because $C$ is full dimensional) is the same as the Krull dimension of $ R$. In general, the length of $K_{\Delta}^{\bullet}$ is equal to the codimension of the smallest dimensional 
 cell appearing in $\Delta$. Example \ref{non-simplicial-conic-complex} shows that the  length can be strictly smaller than the dimension of $R$ because there can be chambers $\Delta$ with no
 zero-dimensional cells. 
  \end{rem}

\begin{ex} \label{non-simplicial-conic-complex}
Consider again the three-dimensional  toric algebra $R$ obtained from the cone over the square in $\mathbb R^3$ as in Example \ref{ConeOverSquare}. This non-simplicial toric ring admits three isomorphism types of conic modules, whose (representative) chambers are pictured below in Figure \ref{fig: chambersC}, with  boundary cells in each  chamber  colored darker. Note that the chamber of type (A) contains a cell of dimension zero (point) but the chambers of types (B) and (C) do not. 
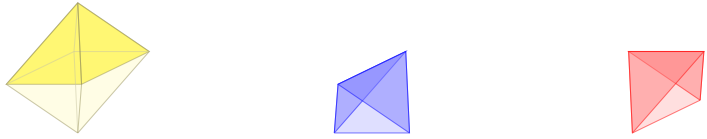
\begin{figure}[h!t]
\centering
\begin{subfigure}[b]{.3\textwidth}
\centering
\begin{tikzpicture}[baseline=(current bounding box.center)]
	\begin{scope}[opacity=.5]
	\node[coordinate] (000) at \ThreeD{0}{0}{0} {};
	\node[coordinate] (001) at \ThreeD{0}{0}{1} {};
	\node[coordinate] (101) at \ThreeD{1}{0}{1} {};
	\node[coordinate] (011) at \ThreeD{0}{1}{1} {};
	\node[coordinate] (111) at \ThreeD{1}{1}{1} {};
	\node[coordinate] (112) at \ThreeD{1}{1}{2} {};
	\draw[dyellow,fill=yellow!25] (000) to (101) to (112) to (011) to (000);
	\draw[dyellow,fill=yellow] (111) to (101) to (112) to (011) to (111);
	\draw[dyellow!50] (101) to (001) to (011);
	\draw[dyellow!50] (000) to (001) to (112);
	\draw[dyellow] (000) to (111) to (112);
	\draw[dyellow] (111) to (112);
\end{scope}
\end{tikzpicture}
\caption{Chamber of module $A_0$}
\end{subfigure}
\begin{subfigure}[b]{.3\textwidth}
\centering
\begin{tikzpicture}[baseline=(current bounding box.center)]
\begin{scope}[opacity=.5]
	\node[coordinate] (011) at \ThreeD{0}{1}{1} {};
	\node[coordinate] (111) at \ThreeD{1}{1}{1} {};
	\node[coordinate] (112) at \ThreeD{1}{1}{2} {};
	\node[coordinate] (122) at \ThreeD{1}{2}{2} {};
	\draw[blue,fill=blue!25] (011) to (111) to (112) to (122) to (011);
	\draw[blue,fill=blue!50] (011) to (122) to (112) to (011);
	\draw[blue,fill=blue!50] (111) to (122) to (112) to (111);
\end{scope}
\end{tikzpicture}
\caption{Chamber of module $A_1$}
\end{subfigure}
\begin{subfigure}[b]{.3\textwidth}
\centering
\begin{tikzpicture}[baseline=(current bounding box.center)]
\begin{scope}[opacity=.5]
	\node[coordinate] (101) at \ThreeD{1}{0}{1} {};
	\node[coordinate] (111) at \ThreeD{1}{1}{1} {};
	\node[coordinate] (112) at \ThreeD{1}{1}{2} {};
	\node[coordinate] (212) at \ThreeD{2}{1}{2} {};
	\draw[red,fill=red!25] (101) to (111) to (112) to (212) to (101);
	\draw[red,fill=red!50] (101) to (112) to (212) to (101);
	\draw[red,fill=red!50] (111) to (112) to (212) to (111);
\end{scope}
\end{tikzpicture}
\caption{Chamber of module $A_2$}
\end{subfigure}
\caption{The three types of chambers with shaded boundary cells.}
\label{fig: chambersC}
\end{figure}
Therefore the complexes $K^{\bullet}_{\Delta}$ corresponding to chambers of type (A) will have length three while the length of the complexes corresponding to any other chamber will have length strictly less. 
The complex $K^{\bullet}_{\Delta}$  corresponding to the first type of chamber is 
\[ A_0 \longrightarrow  A_0^{\oplus4}  \longrightarrow A_1^{\oplus2}\oplus A_2^{\oplus2}  \longrightarrow A_0, \]
with   $M$-graded structure  and mappings given by signed inclusions between chambers, as given in Figure \ref{fig: complexC0}.

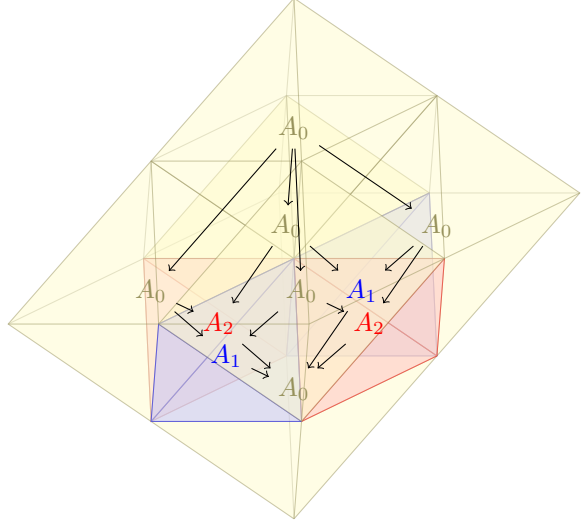
\begin{figure}[h!t]
\begin{tikzpicture}[scale=2]
	\ychamber[\ThreeD{0}{0}{1}]
	\rchamber[\ThreeD{-1}{0}{0}]
	\bchamber[\ThreeD{0}{-1}{0}]
	\ychamber[\ThreeD{0}{0}{0}]
	\ychamber[\ThreeD{1}{0}{1}]
	\ychamber[\ThreeD{0}{1}{1}]
	\ychamber[\ThreeD{1}{1}{2}]
	\rchamber[\ThreeD{0}{0}{0}]
	\bchamber[\ThreeD{0}{0}{0}]
	\ychamber[\ThreeD{1}{1}{1}]
	
	\node[dyellow] (0) at \ThreeD{.5}{.5}{1} {$A_0$};
	\node[blue] (1a) at \ThreeD{.75}{1.25}{1.5} {$A_1$};
	\node[red] (1b) at \ThreeD{.25}{.75}{1.5} {$A_2$};
	\node[blue] (1c) at \ThreeD{.75}{.25}{1.5} {$A_1$};
	\node[red] (1d) at \ThreeD{1.25}{.75}{1.5} {$A_2$};
	\node[dyellow] (2d) at \ThreeD{1.5}{.5}{2} {$A_0$};
	\node[dyellow] (2b) at \ThreeD{.5}{1.5}{2} {$A_0$};
	\node[dyellow] (2c) at \ThreeD{.5}{.5}{2} {$A_0$};
	\node[dyellow] (2a) at \ThreeD{1.5}{1.5}{2} {$A_0$};
	\node[dyellow] (3) at \ThreeD{1.5}{1.5}{3} {$A_0$};

	\draw[->] (1a) to (0);
	\draw[->] (1b) to (0);
	\draw[->] (1c) to (0);
	\draw[->] (1d) to (0);
	\draw[->] (2a) to (1a);
	\draw[->] (2a) to (1d);
	\draw[->] (2b) to (1b);
	\draw[->] (2b) to (1a);
	\draw[->] (2c) to (1c);
	\draw[->] (2c) to (1b);
	\draw[->] (2d) to (1d);
	\draw[->] (2d) to (1c);
	\draw[->] (3) to (2a);
	\draw[->] (3) to (2b);
	\draw[->] (3) to (2c);
	\draw[->] (3) to (2d);
\end{tikzpicture}
\caption{The complex of an octahedral chamber}
\label{fig: complexC0}
\end{figure}
The complex of $R$-modules corresponding to the latter two types of chambers is 
\[ A_2 \longrightarrow A_0^{\oplus2} \longrightarrow A_1 \]
\[ A_1 \longrightarrow A_0^{\oplus2} \longrightarrow A_2, \]
with $M$-graded structure described by signed inclusions between chambers, as given in Figure \ref{fig: complexC12}.
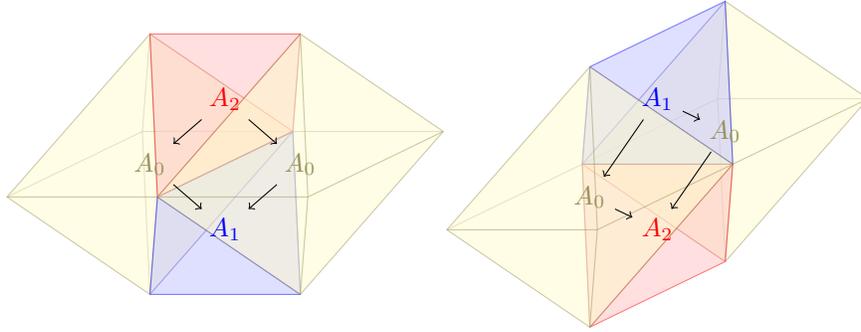
\begin{figure}[h!t]
\begin{subfigure}{.45\textwidth}
\centering
\begin{tikzpicture}[scale=2]
	\ychamber[\ThreeD{0}{1}{1}]
	\bchamber[\ThreeD{0}{0}{0}]
	\rchamber[\ThreeD{0}{1}{1}]
	\ychamber[\ThreeD{1}{1}{1}]
	
	\node[blue] (0) at \ThreeD{.75}{1.25}{1.5} {$A_1$};
	\node[dyellow] (1a) at \ThreeD{.5}{1.5}{2} {$A_0$};
	\node[dyellow] (1b) at \ThreeD{1.5}{1.5}{2} {$A_0$};
	\node[red] (2) at \ThreeD{1.25}{1.75}{2.5} {$A_2$};

	\draw[->] (1a) to (0);
	\draw[->] (1b) to (0);
	\draw[->] (2) to (1a);
	\draw[->] (2) to (1b);
\end{tikzpicture}
\end{subfigure}
\begin{subfigure}{.45\textwidth}
\centering
\begin{tikzpicture}[scale=2]
	\ychamber[\ThreeD{1}{0}{1}]
	\bchamber[\ThreeD{1}{0}{1}]
	\rchamber[\ThreeD{0}{0}{0}]
	\ychamber[\ThreeD{1}{1}{1}]
	
	\node[red] (0) at \ThreeD{1.25}{.75}{1.5} {$A_2$};
	\node[dyellow] (1a) at \ThreeD{1.5}{.5}{2} {$A_0$};
	\node[dyellow] (1b) at \ThreeD{1.5}{1.5}{2} {$A_0$};
	\node[blue] (2) at \ThreeD{1.75}{1.25}{2.5} {$A_1$};

	\draw[->] (1a) to (0);
	\draw[->] (1b) to (0);
	\draw[->] (2) to (1a);
	\draw[->] (2) to (1b);
\end{tikzpicture}
\end{subfigure}
\caption{The complexes of tetrahedral chambers}
\label{fig: complexC12}
\end{figure}
\end{ex}

\subsection{The Acyclicity Lemma.}

Our eventual goal will be to apply the functor $\Hom_R(\mathbb A, -)$ to complexes of the form $K_{\Delta}^{\bullet}$ in order to construct  minimal projective resolutions of simple modules over 
the endomorphism ring  $\End_R(\mathbb A),$ where $\mathbb A$ is a finite sum of conic modules. 
 The following {\bf Acyclicity Lemma}  is crucial.

\begin{lemma}[The Acyclicity Lemma]\label{acyclicity}
Fix a toric algebra $R$ coming from a  pointed full dimensional rational polyhedral  cone $C \subset M_{\mathbb R}$. 
Let $A_\Delta$ and $A_{\Delta'}$ be two conic modules for $R$, with associated  chambers of constancy $\Delta$ and $\Delta'$.
\begin{enumerate}
	\item The complex  $\Hom_{R}(A_{\Delta'},K_{\Delta}^\bullet)$  is    {\it exact} if $A_\Delta \not\simeq A_{\Delta'}$; 
	\item If $A_\Delta \simeq A_{\Delta'}, $ then  $\Hom_{R}(A_{\Delta'},K_{\Delta}^\bullet)$ has one-dimensional homology, which appears (only) in  homological degree zero and  
	 $M$-degree $m$, where $m$ is the unique element of $M$  such that $\Delta = m + \Delta'$.
\end{enumerate}
\end{lemma}

\begin{proof}[Proof of the Acyclicity Lemma]
Because the complex is $M$-graded,  it suffices to prove acyclicity in each degree $m \in M.$
It is enough to prove the Acyclicity Lemma only in degree $m = 0\in M$, because of the 
 identification of $\k$-vector spaces
$$
[\Hom_R(A_{\Delta'}, A_{\tilde \Delta})]_m = [\Hom_R(x^mA_{\Delta'}, A_{\tilde \Delta})]_0 =  [\Hom_R(A_{\Delta' + m}, A_{\tilde \Delta})]_0
$$
for all $m\in M$ and all chambers $\Delta'$, $\tilde \Delta$. 

We prove the Acyclicity Lemma in $M$-degree $0$ in  two cases:
\begin{enumerate}
\item[(i)]  $\Delta = \Delta'$
\item[(ii)] $\Delta \neq \Delta'$
\end{enumerate}

Case (i):  Since $\Delta = \Delta'$ and all the modules appearing in the terms of $K^{\bullet}_{\Delta}$ are of the form $\mathring A_{\tau},$ where $\tau$ is a cell of $\Delta$, we have
$$
A_{\Delta} \supset \mathring A_{\tau};
$$
 the inclusion is {\it proper} unless $\tau$ has codimension zero, in which case $ A_{\Delta} = \mathring A_{\tau}.$
So there are no degree preserving maps $A_{\Delta'} = A_{\Delta} \rightarrow  \mathring A_{\tau}$ except in homological degree zero. That is, the entire chain complex
$[\Hom_R(A_{\Delta'}, K^{\bullet}_{\Delta})]_0$ reduces to the zero complex, except in homological degree zero, where we have $[\Hom_R(A_{\Delta}, A_{\Delta})]_0 = \k$.
This completes the proof of Case (i).

Case (ii): Suppose   $\Delta \neq \Delta'$.  To show that the complex $[\Hom_{R}(A_{\Delta'},K_{\Delta}^\bullet)]_0$  is exact, we first examine this complex of $\k$-vector spaces carefully. 
From Proposition \ref{ConicModProp}(3), we have that 
\begin{equation}\label{4}
\left[\Hom_{R}(A_{\Delta'}, A_{\tilde \Delta})\right]_0 =
 \left\{\begin{array}{cc}
	\mathsf{k} & \text{if} \,\, A_{\Delta'}\subset A_{\tilde \Delta} \\
	0& \text{otherwise.}
	\end{array}\right.
\end{equation}
This  implies 
\begin{lemma}\label{DegreeZeroComplex} 
The zero-th graded piece  $\left[\Hom_R(A_{\Delta'}, K^{\bullet}_{\Delta})\right]_0$ of the $M$-graded complex $\Hom_{R}(A_{\Delta'},K_{\Delta}^\bullet)$ is 
$$
\cdots 
  \longrightarrow 
 \bigoplus_ {\tiny{\begin{array}{c}
{\text{cells }}\tau\in \Delta \\
{\text{codim}}(\tau)=2\\
A_{\Delta'} \subset \mathring A_{\tau} 
  \end{array}} }\k \,\,\,\,\,\,\,\,\,\,  
  \longrightarrow 
 \bigoplus_ {\tiny{\begin{array}{c}
{\text{cells }}\tau\in \Delta \\
{\text{codim}}(\tau)=1\\
A_{\Delta'} \subset \mathring A_{\tau} 
  \end{array}} } \k  \,\,\,\,\,\,\,\,\,\,
  \longrightarrow 
 \bigoplus_
  {\tiny{\begin{array}{c}
{\text{cells }}\tau\in \Delta \\
{\text{codim}}(\tau)=0\\
A_{\Delta'} \subset \mathring A_{\tau} 
  \end{array}} }
   \k  \,\,\,\,\,\,\,\,\,\,
$$
consisting only of those summands indexed by cells  $\tau$ of $\Delta$ for which $A_{\Delta'} \subset  \mathring A_{\tau}$. 
\end{lemma}
Proceeding with the proof of Case (ii), observe that  without loss of generality, we may assume $A_{\Delta} \supset A_{\Delta'}$, because otherwise, there are no degree-preserving maps
$A_{\Delta'} \rightarrow \mathring A_{\tau}$ for $\tau \subset \Delta$,  so  the entire complex is zero. So fix $\Delta$ and $\Delta'$ with $\Delta' \prec \Delta$. We  induce on the quantity $$
\deg \Delta - \deg\Delta' \ ,
$$
where $\deg \Delta$ denotes the degree of the chamber $\Delta$ as defined in Definition \ref{Degree}. Our assumption that $\Delta' \prec \Delta$ but 
$\Delta \neq \Delta'$ ensures that this difference is at least 1. 

\subsubsection{Base Case of Induction}
Assume $\deg \Delta - \deg\Delta' = 1. $  By Lemma \ref{Covering}, it follows that $\Delta' \prec \Delta$ is a saturated chain in the poset of chambers, and  $\Delta$ and $\Delta'$ are on opposite sides of a supporting  hyperplane 
$\mathcal H$ of $\Delta$.
 This means that $A_{\Delta'} = \mathring A_{\tau}$ where $\tau$ is the interior cell of the facet of $\Delta\cap \mathcal H$.  So 
 $A_{\Delta'}$ is a submodule of $A_{\Delta} $ and of $ \mathring A_{\tau}$ but not of any other  $ \mathring A_{\tau'}$ appearing in the chain complex
$K_{\Delta}^\bullet$. So $[\Hom_{R}(A_{\Delta'},K_{\Delta}^\bullet)]_0$ degenerates to the complex with non-zero terms only in homological degrees $0$ and $1$:
$$
[\Hom_{R}(A_{\Delta'}, A_{\Delta} )]_0 \longrightarrow [\Hom_{R}(A_{\Delta'},A_{\Delta'})]_0
$$
both of which are $\k$. This complex is the isomorphism $\k \longrightarrow \k,$ hence exact.

\subsubsection{Inductive Step}
We now fix $d>0$ and assume that we have proved the Acyclicity Lemma in degree $0\in M$ on {\it all } toric varieties and all  chambers  $\Delta' \prec \Delta$ satisfying 
$\deg \Delta - \deg\Delta' < d$.

Given  $\Delta \succ \Delta'$ with $\deg \Delta - \deg\Delta' = d\geq 2$,  use Lemma \ref{Covering} to find a chamber $\Delta''$ such that 
$$
\Delta \succ \Delta''  \succ \Delta'
$$
with the chain $\Delta'' \succ\Delta'$  {\it saturated}. Let $\mathcal H$ be the supporting hyperplane separating the chambers  $\Delta'' $ and $ \Delta'.$

Consider the inclusion of conic modules  $A_{\Delta'} \subset A_{\Delta''}$. The cokernel is annihilated by the ideal $P_{n_t}$ generated by monomials $x^m$ such that 
$\langle m, n_t \rangle > 0$, where $n_t$ is the inward-pointing normal of $\mathcal H$. Thus the cokernel is naturally  a module over the ring $R/P_{n_t}$, the 
toric algebra  $S$ of the cone $C\cap H$ in the subspace $H$, the supporting hyperplane of $C$ with inward pointing normal $n_t$ (parallel to $\mathcal H$).  We claim that this cokernel complex can be understood as the complex of a chamber of constancy over the toric algebra $S$: 

\medskip
\noindent
{\bf Claim:} {\it 
There is a short exact sequence of $M$-graded complexes
$$
0 \longrightarrow
 \Hom_{R} ( A_{\Delta''},K_\Delta^{\bullet}) \hookrightarrow \Hom_{R} (A_{\Delta'},K_\Delta^{\bullet})
 \longrightarrow 
  \Hom_{S} ( A_{\delta''},K_\delta^{\bullet})[1]
   \longrightarrow  0 
$$
where $S$ is the toric algebra of the cone $C\cap H$ and $\delta'', \delta \subset H$ are the  chambers of constancy for 
$S$-modules  as described by Proposition \ref{Restrict}: 
$$
\delta :=  (\Delta-m) \cap  H \,\,\,\,\,\,\,\,\,\delta'' :=  (\Delta''-m) \cap H. 
$$
Here  $m\in \mathcal H\cap M$ is any lattice point used to translate $\mathcal H$ to the subspace $H$, and the  notation $C^{\bullet}[1]$ denotes a homological shift of the complex  $C^{\bullet}.$}

\medskip

Before establishing the claim, we note its sufficiency to complete the proof of  the Acyclicity Lemma. Indeed, since both $\deg \Delta-\deg \Delta''$ and $\deg \delta-\deg \delta''$ are precisely one less than  $\deg \Delta-\deg \Delta',$  the inductive assumption implies that both complexes
$  \Hom_{R} ( A_{\Delta''},K_\Delta^{\bullet})$ and $ \Hom_{S} ( A_{\delta''},K_\delta^{\bullet})$ are exact in degree zero.  Whence so is the complex in the middle. 

To prove the claim, first note that 
by restricting any map 
$A_{\Delta''} \rightarrow K^{\bullet}_{\Delta}$ to $A_{\Delta'}$, we get an inclusion of $M$-graded complexes
of $R$-modules
$$
 \Hom_{R} ( A_{\Delta''},K_\Delta^{\bullet}) \hookrightarrow \Hom_{R} (A_{\Delta'},K_\Delta^{\bullet}).
$$
The cokernel  $\mathcal C^{\bullet}$  is annihilated by the prime ideal 
$P_{n_t}$ and hence is  a complex of modules over the toric algebra $S$.  Its degree zero subcomplex is
\begin{equation}\label{DegZeroCokernel}
\cdots 
 \longrightarrow 
\bigoplus_ {\tiny{\begin{array}{c}
{\text{cells }}\tau\in \Delta \\
{\text{codim}}(\tau)=2\\
A_{\Delta'} \subset \mathring A_{\tau} \\
\tau \subset \mathcal H
  \end{array}} }  \k \,\,\,\,\,\,\,\,\,\, 
 \longrightarrow 
\bigoplus_ {\tiny{\begin{array}{c}
{\text{cells }}\tau\in \Delta \\
{\text{codim}}(\tau)=1\\
A_{\Delta'} \subset \mathring A_{\tau} \\
\tau \subset \mathcal H
  \end{array}} }  \k  \,\,\,\,\,\,\,\,\,\,
 \longrightarrow 
  \bigoplus_
  {\tiny{\begin{array}{c}
{\text{cells }}\tau\in \Delta \\
{\text{codim}}(\tau)=0\\
A_{\Delta'} \subset \mathring A_{\tau} \\
\tau \subset \mathcal H
  \end{array}} } \k  \,\,\,\,\,\,\,\,\,\,
\end{equation}
because of the following lemma:

\begin{lemma}\label{lemma1}
Let $\Delta, \Delta',\Delta''$ be chambers of constancy such that $\Delta\succ \Delta''\succ\Delta'$, with $\Delta'$ and $\Delta''$ adjacent on either side of hyperplane $\mathcal H$. Let $\tau$ be an open cell of $\Delta$.
The natural map 
$$
\left[\Hom_{R}(A_{\Delta''}, \mathring A_{\tau})\right]_0 \longrightarrow \left[\Hom_{R}(A_{\Delta'}, \mathring A_{\tau})\right]_0
$$
induced by restricting maps from $A_{\Delta''}$ to the submodule $A_{\Delta'}$ is either an isomorphism (when $\tau \not\subset \mathcal H$) or the zero map 
(when $\tau \subset \mathcal H$) because the source is zero.
\end{lemma}

\begin{proof} The statement amounts to saying that 
an inclusion $A_{\Delta'} \subset \mathring A_{\tau}$
fails to factor as
$$
A_{\Delta'} \subset  A_{\Delta''}  \subset  \mathring A_{\tau}
$$
if and only if $\tau \subset \mathcal H\cap \Delta.$ 

To see this, note that because $\Delta'$ and $\Delta''$ lie on opposite sides of the single separating hyperplane $\mathcal H$, 
a  monomial $x^m$ is in $ A_{\Delta''}$ but not $ A_{\Delta'}$  if and only if $m\in \mathcal H\cap M$.
So any inclusion  $A_{\Delta'}\subset \mathring A_{\tau}$ will extend to  $A_{\Delta''}$
 unless some $m\in \mathcal H\cap M$ is not in $\mathring C+v$ where $v\in \tau$. This happens if and only if $\mathcal H$ is a supporting hyperplane of the translated cone $C+v$, that is, if and only if 
 $v\in \mathcal H$.   
\end{proof}

Finally, we need to show that the cokernel complex $[\mathcal C^{\bullet}]_0$ is, up to homological shift, the same the complex of conic modules over the toric algebra
 $S = \k[(C\cap H)\cap (M\cap H)]$ described in the claim.
  Pick any $m\in \mathcal H\cap M$, and consider the translation
 $$M_{\mathbb R} \rightarrow M_{\mathbb R} \,\,\,\,\,\,\, {\text{sending}} \,\,\,\,\,\,\, x\mapsto x- m.
 $$
 This map translates $\mathcal H$ to $H$, and respects the partial ordering on  chambers as well as the CW decomposition of $M_{\mathbb R}.$ In particular, 
 the chambers $\Delta \succ \Delta'' \succ \Delta'$ are  translated to chambers
  $\delta \succ \delta'' \succ \delta'$, and every open cell of $\Delta$ is translated to a corresponding open cell of $\delta$. 
 Intersecting with the subspace $H$, we have
 $$
   \delta \cap H  \succ \delta''  \cap H \succ \delta' \cap H = \emptyset
   $$
   and now the chambers $  \delta \cap H  \succ \delta''  \cap H$ are chambers of constancy for the conic modules over $S$ by Proposition \ref{Restrict}; likewise the
   translations of the open cells $\tau \in \mathcal H$ become the open cells of the conic-induced CW decomposition of $H$. Thus the cokernel complex 
   $\left[\mathcal C^{\bullet}\right]_0$ from (\ref{DegZeroCokernel}) above can be identified with the zero-th graded piece of 
   the complex of $S$-modules 
   $$
   \Hom_{S}(A_{\delta''}, K^{\bullet}_{\delta}). 
   $$
 The proof is complete. 
\end{proof}
\section{Global Dimension of Endomorphism Algebras}   \label{Section6}

Throughout this section, $R$ denotes a toric algebra defined by a pointed full dimensional cone $C\subset M_{\mathbb R}$. 
For  any ring $\Lambda$, the {\bf{global dimension} } $\gldim(\Lambda)$ is the maximum projective dimension ({possibly infinite}) of any left or right $\Lambda$-module.

A fundamental theorem of  Auslander-Buchsbaum and Serre states that a {\it commutative} Noetherian local ring $\Lambda$ has finite global dimension if and only if it is regular. When $\Lambda$ is not regular, it may still admit a finitely generated faithful module $M$ 
 such that the ring  $\mathrm{End}_{\Lambda}(M)$ has finite global dimension. Such an endomorphism ring is called a {\bf non-commutative resolution of singularities } of the commutative ring $\Lambda$.

The main result of this section is a constructive proof that toric rings admit natural non-commutative resolutions:

\begin{thm}\label{gldim}
Let $\M = A_{\Delta_1} \oplus \cdots \oplus A_{\Delta_s}$, where the $A_{\Delta_i}$ run through each isomorphism class of conic $R$-modules. Then the global dimension of 
$\mathrm{End}_R(\M)$ is equal to the Krull dimension of $R$.
\end{thm}

As an immediate corollary, we deduce that the same is true for any Morita equivalent ring. In particular, invoking Proposition \ref{ISOM},
\begin{coro}\label{q}
For any natural number $q$,  consider the $R$-module $$R^{1/q} =  \k[\frac{1}{q}M\cap C].$$ Then provided that $q$ is sufficiently large{\footnote{In practice, it is easy to see how large is sufficiently large; see the proof of Proposition \ref{ISOM}.}}, the global dimension of 
$\mathrm{End}_R(R^{1/q})  $ is equal to the dimension of $R$.
\end{coro}

From
Corollary \ref{q}, we deduce also that  in  prime characteristic $p$, the ring of differential operators $D_{\k}(R)$ on a toric ring has finite global dimension; see Theorem \ref{Dmod}. We do not know if the same is true in characteristic zero.

The rest of this section is devoted to the proofs of these results, as well as Theorem \ref{Dmod} on the  finite global dimension of $D_\mathsf{k}(R)$ for toric $R$ over a perfect field of characteristic $p$.

\subsection{Modules over Endomorphism Rings.} We review some of the general module  theory for endomorphism rings in our special setting at hand.
 See also   \cite[4.1]{SmVdB} or \cite[11.1]{Hazewinkeletc}.

Fix an $R$-module $\M$, of the form
\begin{equation}\label{eqM}
 \M := \bigoplus_{i=1}^m A_{\Delta_i}^{d_i}, 
\end{equation}
where the $d_i$ are positive integers and every conic module  is isomorphic to a unique $A_{\Delta_i}$ appearing in the sum. We call such a module $\M$ a \textbf{complete sum of conic modules}.
Our motivating example is $R^{1/q}$, which is isomorphic to a complete sum of conic modules when $q$ is large enough (Proposition \ref{ISOM}). When we choose an isomorphism from  $R^{1/q}$ to a complete sum of conic modules, we make a choice of the representatives $A_{\Delta_i}$ appearing in the sum. The resulting $M$-gradings on $R^{1/q}$ and its endomorphism ring depend on this choice, and so they should be regarded as artificial (but very useful!).

Consider the endomorphism algebra  $\mathrm{End}_R(\M)$. 
Because $\M$ is finitely generated, the ring  $\mathrm{End}_R(\M)$ is Noetherian. Thus the maximal length of a minimal  projective resolution in the category of left and right $\mathrm{End}_R(\M)$-modules is the same, see \cite{AuslanderDimensionIII}. Therefore, we focus attention on the 
{\it right} modules over $\mathrm{End}_R(\M)$.

\subsection{Right modules over $\mathrm{End}_R(\M)$.}
For any $R$-module $B$, $\mathrm{Hom}_R(\M,B)$ is naturally a right $\mathrm{End}_R(\M)$-module with action given by composition on the first factor. This induces a functor 
\[ \mathrm{Hom}_R(\M,-): \mathrm{Mod}(R) \longrightarrow \mathrm{Mod}(\mathrm{End}_R(\M))\]
where  $\mathrm{Mod}(\mathrm{End}_R(\M))$ denotes the category of { right} $\mathrm{End}_R(\M)$-modules. 
This functor restricts to an equivalence between the full subcategories
\begin{equation}\label{equiv} 
\mathrm{add}(\M) \garrow{\sim} \mathrm{proj}(\End_R(\M))  
\end{equation}
of finite direct sums of (graded) summands of $\M$ and 
finitely generated (graded) projective right   $\mathrm{End}_R(\M)$ modules.
To see this, observe that the functor $G_*:=\mathrm{Hom}_R(\M,-)$ has left adjoint $G^*:=-\otimes_{\mathrm{End}_R(\M)}\M$, and that the adjunction maps define isomorphisms
\[ G^*G_*(\M)\garrow{\sim}\M \quad \text{ and } \quad \mathrm{End}_R(\M)\garrow{\sim} G_*G^*(\mathrm{End}_R(\M)) \ . \]
Both functors respect the $M$-grading and commute with direct sums, so they induce an equivalence $\mathrm{add}(\M)\simeq \mathrm{add}(\mathrm{End}_R(\M))$, the latter of which equals $\mathrm{proj}(\mathrm{End}_R(\M))$.
By restricting to the $M$-graded category, we have a functorial bijection between  conic $R$-modules and indecomposable projective $M$-graded modules over $\mathrm{End}_R(\M)$. Summarizing and continuing, we have

\begin{prop}\label{proj}
  For any $\Delta$, 
\begin{enumerate}
\item 
The $\mathrm{End}_R(\M)$-module
\[ P_\Delta := \mathrm{Hom}_R(\M,A_\Delta)\]
is a graded indecomposable projective module, and every graded indecomposable projective $\mathrm{End}_R(\M)$-module is isomorphic to a module of this form.
\item \label{ProjRad}
The projective module $P_\Delta$ has a unique maximal graded submodule.  Every graded simple $\mathrm{End}_R(\M)$-module is isomorphic to a quotient of some $P_\Delta$ by this maximal submodule.
\end{enumerate}
\end{prop}

\noindent We let $S_\Delta$ denote the unique graded simple quotient of $P_\Delta$.

Although Proposition \ref{proj} can be found in the literature (eg.~see \cite[Lemma 4.1.1]{SmVdB} or \cite[Prop.~11.1.1]{Hazewinkeletc}),  we keep the paper self-contained and our proof easy to follow by explicitly describing the unique maximal (right) submodule of $P_\Delta$--called the \emph{radical} of $P_\Delta$--in Section \ref{RadicalTheory} below.

\begin{rem}  \label{gradingRemark}
The word {\it graded}  in Proposition \ref{proj} above refers to the natural $M$-grading on all the objects;  all maps are homogeneous in this grading. 
However, because $C$ is pointed,  the toric ring $R$ can also be given a (non-canonical) connected $\mathbb Z$-grading by fixing a group homomorphism $M\rightarrow \mathbb{Z}$  taking strictly positive values on $M\cap(C\smallsetminus\{0\})$.  This grading induces compatible gradings on conic modules, on $\M$, and on $\mathrm{End}_R(\M)$. Both statements in Proposition \ref{proj} are valid for this connected $\mathbb Z$-grading as well. 
The same is true of Theorem \ref{projective} below.
 \end{rem}

The following theorem (proven in the next section) states that the functor $G_*$ sends the complex $K^\bullet_\Delta$ to a projective resolution of the simple quotient module $S_\Delta$.

\begin{thm}\label{projective}
The complex $\mathrm{Hom}_R(\M,K^\bullet_\Delta)$ is a graded $\mathrm{End}_R(\M)$-projective resolution of the simple module $S_\Delta$, and this projective resolution has minimal length.
\end{thm}

\subsection{Radicals} \label{RadicalTheory}
We review some general theory in preparation to prove Theorem \ref{projective}. By definition, a homogeneous $R$-module homomorphism $f:A_\Delta\rightarrow A_{\Delta'}$ is \textbf{radical} if it  is {\it not} an isomorphism.
The radical homomorphisms generate a graded $R$-submodule
 $\mathrm{Rad}(A_{\Delta},A_{\Delta'})\subseteq \mathrm{Hom}_R(A_\Delta,A_{\Delta'})$. By Proposition \ref{Rad},  
\[ \mathrm{Rad}(A_\Delta,A_{\Delta'}) = \Span\{x^m \mid m\in M \text{ such that } \Delta +  m  \prec \Delta'\}, \]
where the symbol $\prec$ indicates {\it strictly less} in the poset on chambers. In particular, $ \mathrm{Hom}_R(A_\Delta,A_{\Delta'})/\mathrm{Rad}(A_{\Delta},A_{\Delta'})$ is one dimensional  over $\k$
if $A_{\Delta} \simeq  A_{\Delta'}$ and zero otherwise.
 
  More generally, we may represent a graded homomorphism
  from a  direct sum $\M$ of $n$ conic modules  $A_i$ to  a  direct sum $\mathbb B$ of $m$ conic modules  $B_j$ by
 an $m\times n$  matrix $[\phi_{ij}]$ where $\phi_{ij} \in \mathrm{Hom}_R(A_j, B_i)$;
 such a homomorphism is  \textbf{radical}  if each of its entries $\phi_{ij}$ is.\footnote{Here, we consider the elements of $\mathbb A$ and $\mathbb B$ as column vectors with entries in $A_i$ and $B_j$ respectively,  and the matrix acts on the left.}
  In particular,  the elements of   $\mathrm{Hom}_{\End_R \M}(\mathbb A, A_{\Delta}) $ are represented by row vectors with entries in $ \mathrm{Hom}_{R}(A_i, A_{\Delta})$; the homogenous maps  in its radical have entries in the radical of each 
   $ \mathrm{Hom}_{R}(A_i, A_{\Delta})$.  
   This is clearly a maximal graded right submodule: if some  homogeneous row matrix $(\phi_j)$ is not in $ \mathrm{Rad}(\M, A_{\Delta})$, then some entry $\phi_j\in  \mathrm{Hom}_{R}(A_{j}, A_{\Delta})$ is a graded isomorphism. Since elements of $\mathrm{End}_R(\M)$ act on the right by  column operations on the row vector  $(\phi_j)$,  the right module generated by $(\phi_j)$ contains the 
   row vector  $(\psi_j) \in  \mathrm{Hom}_{R}(\M, A_{\Delta})$ which has zeros in every spot except the $j$-th, which is the identity map. Since this idempotent is a $\mathrm{End}_R(\M)$ generator for the right module $\mathrm{Hom}_{\End_R(\M)}(\mathbb A, A_{\Delta}), $ we conclude that $ \mathrm{Rad}(\M, A_{\Delta})$ is maximal.

    Put differently, $\mathrm{End}_R(\M)$ is \emph{graded semiperfect} in the language of  \cite{Dascalescu}.


\begin{proof}[Proof of Theorem \ref{projective}]
Fix a chamber $\Delta$, and corresponding simple $\mathrm{End}_R(\M)$-module $S_{\Delta}$.  
Recall that $K_\Delta^0=A_\Delta$ by construction. It follows that the $\mathrm{Hom}_R(\M,K_\Delta^0)=P_\Delta$, and so there is a natural surjection $\mathrm{Hom}_R(\M,K_\Delta^0) \rightarrow S_\Delta$.

The natural surjection $\mathrm{Hom}_R(\M,K_\Delta^0) = P_{\Delta} \rightarrow S_\Delta$ makes the complex $\mathrm{Hom}_R(\M,K^\bullet_\Delta)$ into a projective resolution of $S_\Delta$.
Indeed, since each $K_\Delta^i$ is a direct sum of conic modules, each  $\mathrm{Hom}_R(\M,K^i_\Delta)$ is a projective module over $\mathrm{End}_R(\M)$.  It remains to show that this complex is \emph{exact}, and that there is no shorter projective resolution.

The acyclicity Lemma \ref{acyclicity} tells us that this complex  $\mathrm{Hom}_R(\M,K^i_\Delta)$  is exact except possibly at the penultimate spot
\[ \bigoplus_{{\tiny{\begin{array}{c}
{\text{cells }}\tau\in \Delta \\
{\text{codim}}(\tau)=1\\
  \end{array}} }}\Hom_R(\M,\mathring{A}_\tau) \xrightarrow{} \Hom_R(\M,A_\Delta)=P_\Delta \xrightarrow{} S_\Delta \ . \]
  Since the kernel of the natural projection map $P_\Delta \xrightarrow{} S_\Delta$ is the radical of $P_\Delta$, we need to check that this is the image of the incoming map. But this is immediate from  Lemma \ref{acyclicity}(2).

It remains to show that $S_{\Delta}$ has no shorter projective resolution. This follows from the following proposition, whose proof will appear below:

\begin{prop}\label{Ext}
For any pair of chambers of constancy $\Delta$ and $\Delta'$,
\[ \dim(\mathrm{Ext}^i_{\mathrm{End}_R(\M)}(S_{\Delta},S_{\Delta'})) = \#\text{ of codim. $i$ cells $\tau$ in $\Delta$ with $\mathring{A}_\tau\simeq A_{\Delta'}$}
\]
\end{prop}

To see that $\mathrm{Hom}_R(\M,K^\bullet_\Delta)$   is a minimal projective resolution of $S_{\Delta}$, let $\tau$ be a cell of maximal codimension in $\Delta$, and let $\Delta'$ be the chamber of constancy such that $A_{\Delta'}=\mathring{A}_\tau$. The length of $\mathrm{Hom}_R(\M,K_\Delta^\bullet)$ is $\mathrm{codim}(\tau)$, and from the proof of Proposition \ref{Ext} it follows that 
\[\mathrm{Ext}^{\mathrm{codim}(\tau)}_{\mathrm{End}_R(\M)}(S_\Delta,S_{\Delta'}) \neq 0 \ . \]
Therefore, $S_\Delta$ cannot have a projective resolution shorter than length $\mathrm{codim}(\tau)$.
\end{proof}

\begin{proof}[Proof of Proposition \ref{Ext}]
Using the resolution $\mathrm{Hom}_R(\M,K_\Delta^\bullet)\rightarrow S_\Delta$ to compute the Ext groups,  
 this Ext group is the $i$th cohomology of the complex 
\[ \mathrm{Hom}_{\mathrm{End}_R(\M)}(\mathrm{Hom}_R(\M,K_\Delta^\bullet),S_{\Delta'}) \]
By Lemma \ref{pands} below, this is equal to
\[
\bigoplus_{\tau \in \Delta}\mathrm{Hom}_R(\mathring{A}_\tau,A_{\Delta'})/\mathrm{Rad}(\mathring{A}_\tau,A_{\Delta'})\]
as a graded vector space. On each component, the boundary map is given by a sum of compositions with proper homogeneous injections, and so consequently every component of the image is $0$. Hence, the boundary map is zero on the above graded vector space. The subspace of cohomological degree $i$ corresponds to $\tau$ of codimension $i$. As the $\tau$-indexed summand is $1$-dimensional if $\mathring{A}_\tau\simeq A_{\Delta'}$ and is $0$ otherwise, the dimension of Ext counts the desired $\tau$. This will conclude the proof of Proposition \ref{Ext} as soon as we prove Lemma \ref{pands}. \end{proof}

\begin{lemma}\label{pands}
For any $\Delta$ and $\Delta'$, there is a graded $\k$-vector space isomorphism
\[ \mathrm{Hom}_{\mathrm{End}_R(\M)}(P_\Delta,S_{\Delta'}) \simeq \mathrm{Hom}_R(A_\Delta,A_{\Delta'})/\mathrm{Rad}(A_\Delta,A_{\Delta'})\] 
These spaces are dimension  one or zero,  depending on whether $A_\Delta\simeq A_{\Delta'}$ or not.
\end{lemma}

\begin{proof}
By the projectivity of $P_\Delta$, there is a short exact sequence
\[ 0 \rightarrow  \mathrm{Hom}(P_\Delta,\mathrm{Rad}(P_{\Delta'}))
\rightarrow  \mathrm{Hom}(P_\Delta,P_{\Delta'})
\rightarrow  \mathrm{Hom}(P_\Delta,S_{\Delta'})
\rightarrow 0\]

The equivalence of categories (\ref{equiv}) induced by $\mathrm{Hom}_R(\M,-)$  guarantees that every homogenous map $g:P_\Delta\rightarrow P_{\Delta'}$ is given by applying $\mathrm{Hom}_R(\M,-)$ to a homogenous map $g':A_\Delta\rightarrow A_{\Delta'}$. The composition $\M\garrow{f}A_\Delta\garrow{g'} A_{\Delta'}$ is in $\mathrm{Rad}(P_{\Delta'})$ for all $f$ if and only if $g'\in \mathrm{Rad}(A_\Delta,A_{\Delta'})$. Hence, the isomorphism $\mathrm{Hom}(A_\Delta,A_{\Delta'})\simeq \mathrm{Hom}(P_{\Delta},P_{\Delta'})$ induces a bijection between $\mathrm{Rad}(A_\Delta,A_{\Delta'})$ and the image of $\mathrm{Hom}(P_\Delta,\mathrm{Rad}(P_{\Delta'}))$. By the above short exact sequence, $\mathrm{Hom}(P_\Delta,S_{\Delta'})\simeq \mathrm{Hom}(A_\Delta,A_{\Delta'})/\mathrm{Rad}(A_\Delta,A_{\Delta'})$.
\end{proof}

\begin{rem}
The resolution $\mathrm{Hom}_R(\M,K_\Delta^\bullet)$ is `minimal' in a stronger sense. Specifically, it lifts to a homogeneous split injection $\mathrm{Hom}_R(\M,K_\Delta^\bullet)\hookrightarrow P^\bullet$ along any graded projective resolution $P^\bullet\rightarrow S_\Delta$. This is the graded analog of a `minimal projective resolution', and can be proven by observing the morphisms in $\mathrm{Hom}_R(\M,K_\Delta^\bullet)$ have coordinates in $\mathrm{Rad}(\mathrm{End}_R(\M))$ in an appropriate sense.
\end{rem}

\begin{coro}\label{pdim}
The projective dimension of $S_\Delta$ is the maximum codimension of cells in $\Delta$. In particular, $\mathrm{pdim}(S_\Delta)\leq \mathrm{rank}(M)$, with equality when $A_\Delta$ is free.
\end{coro}

\begin{ex}\label{non-simplicial}
Consider again the three-dimensional toric algebra $R$ coming from the cone over the square in $\RR^3$, cf.~Example \ref{ConeOverSquare}. We have already calculated the three types of conic complexes for the conic modules $A_0, A_1, A_2$. Now consider $\M=A_0 \oplus A_1 \oplus A_2$ and $\End_R(\M)$. This ring has three graded simples $S_{A_i}=\Hom_R(\M,A_i)/\mathrm{Rad}(\M,A_i)$, $i=0,1,2$. We obtain projective resolutions for the simple modules  $S_{A_i}$ by applying the functor  $\Hom_{R}(\M,  -)$ to the conic complexes $K^{\bullet}_{\Delta_i}$. Thus we see that  
the endomorphism ring $\End_{R}(\M)$ has two graded simple modules of projective dimension $2$ defined by the tetrahedral chambers, and one graded simple of projective dimension $3$ defined by the octahedral chamber.
The same is true for the Morita equivalent ring $\Hom_{R}(R^{1/q}, R^{1/q})$, for $q\gg 0$.
\end{ex}

We can finally prove the opening theorem of the section.
\begin{proof}[Proof of Theorem \ref{gldim}]
The bound $\gldim(\mathrm{End}_R(\M))\geq \mathrm{rank}(M)$ follows from Corollary \ref{pdim}.

Since $\mathrm{End}_R(\M)$ is a finitely generated module over a commutative Noetherian ring $R$, a theorem of Auslander \cite{AuslanderDimensionIII} implies the left and right global dimensions coincide, and a subsequent theorem of Bass \cite[III, Prop.~6.7]{Bas68} implies the following.
\[ \gldim(\mathrm{End}_R(\M)) = \sup\{ \mathrm{pdim}_{\mathrm{End}_R(\M)}(S) \mid S\text{ is a simple right $\mathrm{End}_R(\M)$-module}\}\]
The remainder of the proof will reduce to \emph{graded} simple right $\mathrm{End}_R(\M)$-modules.

Let $S$ be a simple right $\mathrm{End}_R(\M)$-module (with no grading). The annihilator $\mathrm{Ann}(S)$ in $\mathrm{End}_R(\M)$ is a two-sided maximal ideal, and so $\mathfrak{m}:=\mathrm{Ann}(S)\cap R$ is a maximal ideal in $R$. This implies $R/\mathfrak{m}$ is a field, and $S$ is a finite dimensional vector space over $R/\mathfrak{m}$. Since $R/\mathfrak{m}$ must be a finite extension of the ground field $\mathsf{k}$, $S$ must also be finite dimensional over $\mathsf{k}$.

Next, fix a connected $\mathbb{Z}$-grading on $R$, which induces a $\mathbb{Z}$-grading on $\mathrm{End}_R(\M)$ (see Remark \ref{gradingRemark}). This $\mathbb{Z}$-grading may be weakened to a descending $\mathbb{Z}$-filtration, by setting 
\[\mathcal{F}_i(\mathrm{End}_R(\M)):= \bigoplus_{j\geq i} \mathrm{End}_R(\M)_{\deg=j}\]
Note that the associated graded algebra $gr(\mathrm{End}_R(\M))$ is canonically identified with $\mathrm{End}_R(\M)$.

Choose a surjection $\mathrm{End}_R(\M)\rightarrow S$, which induces a compatible filtration on $S$. Since the filtrations on $\mathrm{End}_R(\M)$ and $S$ are \emph{left-limited}, a theorem of \cite[Cor.~7.6]{GradedRings} implies that 
\[ \mathrm{pdim}_{\mathrm{End}_R(\M)}(S) \leq \mathrm{pdim}_{\mathrm{End}_R(\M)}(gr(S)) \ . \]
The associated graded module $gr(S)$ is still finite dimensional over $\mathsf{k}$, and so it has a finite composition sequence as a $\mathbb{Z}$-graded $\mathrm{End}_R(\M)$-module, whose subquotients are $\mathbb{Z}$-graded simple $\mathrm{End}_R(\M)$-modules. As the projective dimension of $gr(S)$ is bounded by the maximum projective dimension of the subquotients (see e.g. the proof of Bass' theorem \cite[III, Prop.~6.7]{Bas68}), we deduce that 
\[ \mathrm{pdim}(S) \leq \mathrm{pdim}(gr(S)) \leq \sup_{\Delta} \{ \mathrm{pdim}(S_\Delta) \}= \mathrm{rank}(M) \]
where the last equality is Corollary \ref{pdim}.
\end{proof}

\begin{rem}
Our results carry over to the complete case: if  $\widehat R$ is the completion of  a pointed toric algebra $R$ at its  maximal graded ideal, then $\End_{\widehat R} (\widehat{R} \otimes \mathbb A)$ has global dimension equal to the dimension of $\widehat R$.
This follows from a general properties about the behavior of global dimension under completion; cf.~e.g.~\cite[3.4]{SVdB2}.
\end{rem}

\subsection{Global dimension of differential operators in positive characteristic}

Given a commutative algebra $R$ over a field $\k$, the algebra of \textbf{differential operators on $R$ over $\k$} is defined as (see e.g.~the discussion in \cite[2.1]{SmVdB})
\[ D_\k(R) := \bigcup_{n\in \mathbb{N}} D_\k^n(R) \]
where $D_\k^0(R) := R$ and 
\begin{align*}
D_\k^n(R) &:= \{ \theta \in \mathrm{End}_\k(R) \mid \text{for all }r\in R, [r,\theta] \in D_\k^{n-1}(R) \}  \ .
\end{align*}
If $R$ is the coordinate ring of a smooth complex variety, then $D_\mathbb{C}(R)$ is the classical ring of regular differential operators generated by multiplication operators and derivatives along vector fields.

When $\k$ has positive characteristic, the ring $D_\k(R)$ exhibits several behaviors without a natural geometric analog. Most relevant for us is the following.
\begin{lemma}\cite{Yekutieli}
Let $\k$ be a perfect field of characteristic $p$, and let $R$ be finitely generated as an $R^p$-module.
Then, as subalgebras of $\mathrm{End}_\k(R)$,
\[ \bigcup_{e\in \mathbb{N}} \mathrm{End}_{R^{p^e}}(R) = D_\k(R) \ . \]
\end{lemma}
\noindent That is, every $R^{p^e}$-linear map from $R$ to itself is a differential operator, and every differential operator is $R^{p^e}$-linear for sufficiently large $e$.
Note that $D_\k^e(R) $ almost never equals $\mathrm{End}_{R^{p^e}}(R)$, and $D_\k^e(R)$ is almost never an algebra.

\begin{thm} \label{Dmod}
Let $R$ be a pointed  toric ring of dimension $d$ over a perfect  field $\k$ of prime characteristic $p$. Then the ring $D_\k(R)$ of differential operators on $R$ has finite global dimension. 
\end{thm}

\begin{proof} Let $R^{1/p^e}$ denote the ring of $p^e$-th roots of elements in $R$. The Frobenius map  $F^e: R\rightarrow R^{p^e}$ induces an isomorphism 
\[ \mathrm{End}_R(R^{1/p^e})
\stackrel{\sim}{\longrightarrow} 
 \mathrm{End}_{R^{p^e}}(R) \ . \]
Consequently, Corollary \ref{q} implies that, for $e$ sufficiently large, 
\[ \gldim(\mathrm{End}_{R^{p^e}}(R)) = \gldim(\mathrm{End}_{R}(R^{1/p^e})) = d \ . \]
By a theorem of Ber{\v{s}}te{\u\i}n \cite{Berstein}, 
\[ \rgldim(D_\k(R)) = \mathrm{rgldim}\left(\bigcup_{e\in \mathbb{N}} \mathrm{End}_{R^{p^e}}(R)\right) \leq 1+ \lim_{e\in N}\left(\mathrm{rgldim}\left(\mathrm{End}_{R^{p^e}}(R)\right)\right) = d+1\]
A symmetric application of Ber{\v{s}}te{\u\i}n's theorem implies that $\lgldim(D_\k(R))\leq d+1$ as well, implying that $\gldim(D_\k(R))\leq d+1$.
\end{proof}

\begin{rem}
The proof implies that $\gldim(D_\k(R))\leq d+1$. We suspect that in fact  $\gldim(D_\k(R))= d$, but as yet have not proved it. 
If $R$ is smooth over $\k$, then Paul Smith's theorem \cite{PSmithCharpRegular} implies that $\gldim(D_\k(R))=d$. 
The general case would follow if we could show that for $q_0 < q$ sufficiently large, then $\End_R(R^{1/q})$ is flat (as a left or right) module  over  $\End_R(R^{1/q_0})$.\end{rem}

\section{Non-commutative crepant resolutions} \label{Sec7}

Non-commutative {\it crepant} resolutions (NCCRs) were introduced by Michel Van den Bergh in his foundational paper  \cite{VanDenBergh}.  A natural question arises now that we have shown $\End_R(\M)$ is a non-commutative resolution: when are these resolutions crepant?

First we recall the definition.
\begin{defn}\label{NCCRdef}
Let $S$ be a normal Cohen-Macaulay domain of dimension $d$, and let $B$ be a finitely generated reflexive $S$-module. The algebra $\mathrm{End}_S(B)$ is a \textbf{non-commutative crepant resolution (NCCR)} of $S$ if every simple   $\mathrm{End}_{S}(B)$-module has projective dimension $d$.
\end{defn}

Van den Bergh first defined NCCRs only for Gorenstein normal domains $S$; in this case, he showed that $\End_S(B)$ is an NCCR if and only if  $\End_S(B)$ has finite global dimension and is Cohen-Macaulay as an $S$-module; see \cite[Lemma 4.2]{VanDenBergh}. 
 If we assume only that $S$ is a normal Cohen-Macaulay domain, Iyama and Wemyss showed that Van den Bergh's alternate characterization of an NCCR continues to be valid, as long as $S$ admits a canonical module \cite[2.17]{IyamaWemyss}.
For a discussion of the justifications and applications of the larger Cohen-Macaulay generality, consult \cite{IyamaWemyss,DaFI}. See also \cite{Leuschke} for a survey on NCCRs.

The main result of this section is the following characterization of when $\End_R(\M)$ is crepant:
\begin{thm}\label{simplicial}
For a toric algebra $R$ and a complete sum of conic modules $\M$ (such as $R^{1/q}$ for large enough $	q$), $\End_R(\M)$ is an NCCR of $R$ if and only if $R$ is a simplicial toric algebra.
\end{thm}
\noindent This theorem will be proven in two parts, as Propositions \ref{simplicialNCCR} and \ref{nonsimplicialnonNCCR}.

\subsection{Simplicial toric algebras} 

Much of the story of this paper simplifies significantly in the simplicial case. 
\begin{prop}\label{simp}
If $R$ is simplicial, then every chamber of constancy $\Delta\subset M_\mathbb{R}$ contains a unique $0$-cell.
\end{prop}
\begin{proof}
Since $|\Sigma_1|=\dim(M_\mathbb{R})$, a $0$-cell in the CW decomposition of $M_\mathbb{R}$ is a point $v$ such that $\langle v,n_i\rangle \in \mathbb{Z}$ for all $i$. For this $0$-cell to be in $\Delta$, we must have $\langle v,n_i\rangle = d_i$ for all $i$, where $d_i := \lceil\langle w,n_i\rangle\rceil$ for any $w\in \Delta$. Since $\Sigma_1$ is a basis, there is a unique solution $v$ to this system.
\end{proof}

\begin{lemma} \label{Lem:simplicial}
Let $R$ be a simplicial toric algebra. For all $\Delta, \Delta'$, 
\[ \mathrm{Hom}_R(A_\Delta,A_{\Delta'}) = A_{w-v} \]
where $v\in \Delta$ and $w\in \Delta'$ are the unique $0$-cells in the respective chambers of constancy.
\end{lemma}
\begin{proof}
If $u\in C+(w-v)$, then $u+(C+v)\subset C+w$. Therefore, for any $x^u\in A_{w-v}$, $x^uA_v\subset A_w$ and so $x^u\in \mathrm{Hom}_R(A_\Delta,A_{\Delta'})$.

Consider $u\in M$ such that $x^u\in \mathrm{Hom}_R(A_\Delta,A_{\Delta'})$. 
For each $n_i\in \Sigma_1$, there exist $m\in M\cap (C+v)$ and $m'\in M\cap (C+w)$ such that $\langle m,n_i\rangle = \langle v,n_i\rangle$ and $\langle m',n_i\rangle = \langle w,n_i\rangle$. Since $x^u\cdot x^m = x^{m+u}\in A_w$, 
\[ \langle m+u,n_i\rangle \geq \langle w,n_i\rangle 
\Rightarrow \langle v,n_i\rangle + \langle u,n_i\rangle \geq \langle w,n_i\rangle 
\Rightarrow \langle u,n_i\rangle \geq \langle w-v,n_i\rangle \ . \]
Therefore, $x^u\in A_{w-v}$.
\end{proof}


\begin{prop} \label{simplicialNCCR}
Let $R$ be a simplicial toric algebra of Krull dimension $d$, and let $\M$ be a complete sum of conic modules. Then $\mathrm{End}_R(\M)$ is an NCCR of $R$. 
\end{prop}

\begin{proof}
Assume $R$ is simplicial. Then by Lemma \ref{Lem:simplicial}, $\End_R(\M)$ is a sum of conic modules, hence Cohen-Macaulay. On the other hand, each chamber of constancy contains a zero-cell, so all \emph{graded} simple right $\End_R(\M)$-modules have projective dimension $d$ by Corollary \ref{pdim}. We show that $\End_R(\M)$ is homologically homogeneous, that is, \emph{any} simple $S$ of $\End_R(\M)$ has projective dimension $d$. Therefore, assume that $S$ has projective dimension $k < d$. This means that we can find an $\End_R(\M)$-projective resolution
\[ 0 \xrightarrow{} P_k \xrightarrow{} \cdots \xrightarrow{} P_0 \xrightarrow{} S \xrightarrow{} 0 \ ,\]
where each $P_i$ is a finitely generated projective $\End_R(\M)$-module.
Since by Lemma \ref{Lem:simplicial} $\End_R(\M)$ has depth $d$ as an $R$-module, any $R$-direct summand must also have depth equal to $d$. 
Thus, since any $P_i$ is a finite direct sum of direct summands of $\End_R(\M)$, it follows that $\mathrm{depth}_R(P_i )=d$ for all $i=0, \ldots, k$. With the same argument as in the proof of Theorem \ref{gldim} we see that $S$ has finite length as an $R$-module, which implies that $\mathrm{depth}_R(S)=0$. But the behavior of depth in exact sequences is well-understood: this implies $\mathrm{depth}_R(P_k)=k <d$ (see e.g., \cite[Cor.18.6]{Eis95}), a contradiction. 
\end{proof}

\begin{rem}
When $R$ is simplicial, $R$ is isomorphic to the ring of invariant polynomials $\mathsf{k}[x_1,\ldots , x_d]^G$ for the action of a finite abelian group $G$ on $\mathsf{k}^d$. As an $R$-module, the ring $\M=\mathsf{k}[x_1,\ldots , x_d]$ is a complete sum of conic modules, and the corresponding endomorphism ring $\mathrm{End}_R(\M)$ is isomorphic to the skew group ring $\mathsf{k}[x_1,\ldots , x_d]*G$. In \cite[Section 4]{Craw-Quintero-Velez}, so-called \emph{cellular resolutions} are constructed for these skew group rings. However, it is not clear how to extend their construction to the non-simplicial case.
\end{rem}


\begin{ex}
Let $R$ be a singular toric algebra of dimension $2$. Without loss of generality (see \cite[2.2]{Fulton}), the cone $C$ can  be taken to be generated by $e_1$ and $ke_1 + me_2$, with $0 < k < m$ and $(k,m)=1$, and where $e_i$ are the two standard basis vectors of $M  = \mathbb Z^2$. Thus $C$ is simplicial and the inner normals are $n_1=e_2$ and $n_2=me_1-ke_2$. 

Let us first determine all isomorphism classes of conic modules: since the half-open square $(-1,0]^2$ is a fundamental domain for the action of $M$, we find a chamber of constancy for each isomorphism class in it, by Proposition \ref{isom}.
There are $m$ different chambers,  which can be indexed by their unique  $0$-cells $v_i=-\frac{i}{m}e_1$ for $0 \leq i <m$ (Proposition \ref{simp}). In particular, the chamber of constancy $\Delta_i:=\Delta_{v_i}$ consists of all $x=(x_1,x_2)$ such that
\[ -1 < x_2 \leq 0 \quad \text{ and } -i-1 < mx_1 -kx_2 \leq -i \ . \]
Denote the corresponding conic module by $A_{i}$. A general chamber of constancy can be indexed by its $0$-cell $v_{ij}=(-\frac{i}{m},j)$ for $i,j \in \ZZ$. The corresponding conic module is then $A_{ij}:=A_{v_{ij}}$. Now we are ready to describe the conic complex $\mathcal{K}^\bullet_{\Delta_{i}}$: it is of the form
\begin{equation} \label{conic-res-dim2}A_{i-k-1,1} \longrightarrow{} A_{{i-1}}\oplus A_{i-k,1} \longrightarrow{}  A_{i}   \ ,
 \end{equation}
with the first index$\mod m$. It is easy to see that $A_{i-k,1} \simeq A_{i-k}$ and $A_{i-k-1,1} \simeq A_{i-k-1}$. \\ 
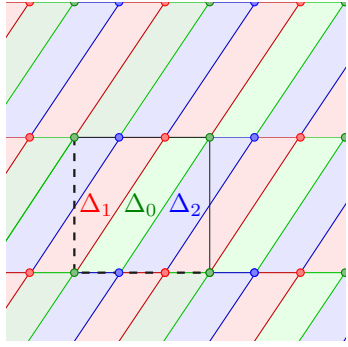
\begin{figure}[h!]
\begin{tikzpicture}[baseline=(current bounding box.center),scale=1.8]
	\clip (-1.5,-1.5) rectangle (1.,1.);
		\begin{scope}[xshift=1.66cm]
	\path[fill=red!10] (-1,-1) to (-0.66,-1) to (0,0) to (0.66,1) to (1,1) to (0.33,0) to (0,0) to (-0.33,0) to (-1,-1);
	\draw[red!75!black] (-0.66,-1) to (0,0) to (-0.33,0);
	\draw[red!75!black] (0.66,1) to (1,1) to (0.33,0);
	\end{scope}
	\begin{scope}[xshift=1.33cm]
	\path[fill=blue!10] (-1,-1) to (-0.66,-1) to (0,0) to (0.66,1) to (1,1) to (0.33,0) to (0,0) to (-0.33,0) to (-1,-1);
		\draw[blue!75!black] (-0.66,-1) to (0,0) to (-0.33,0);
	\draw[blue!75!black] (0.66,1) to (1,1) to (0.33,0);
	\end{scope}
	\begin{scope}[xshift=1cm]
	\path[fill=green!10] (-1,-1) to (-0.66,-1) to (0,0) to (0.66,1) to (1,1) to (0.33,0) to (0,0) to (-0.33,0) to (-1,-1);
	\draw[green!75!black] (-0.66,-1) to (0,0) to (-0.33,0);
	\draw[green!75!black] (0.66,1) to (1,1) to (0.33,0);
	\end{scope}
	\begin{scope}[xshift=0.66cm]
	\path[fill=red!10] (-1,-1) to (-0.66,-1) to (0,0) to (0.66,1) to (1,1) to (0.33,0) to (0,0) to (-0.33,0) to (-1,-1);
		\draw[red!75!black] (-0.66,-1) to (0,0) to (-0.33,0);
	\draw[red!75!black] (0.66,1) to (1,1) to (0.33,0);
	\end{scope}
	\begin{scope}[xshift=0.33cm]
	\path[fill=blue!10] (-1,-1) to (-0.66,-1) to (0,0) to (0.66,1) to (1,1) to (0.33,0) to (0,0) to (-0.33,0) to (-1,-1);
	\draw[blue!75!black] (-0.66,-1) to (0,0) to (-0.33,0);
	\draw[blue!75!black] (0.66,1) to (1,1) to (0.33,0);
	\end{scope}
	\begin{scope}[yshift=0cm]  
	\path[fill=green!10] (-1,-1) to (-0.66,-1) to (0,0) to (0.66,1) to (1,1) to (0.33,0) to (0,0) to (-0.33,0) to (-1,-1);	
		\draw[green!75!black] (-0.66,-1) to (0,0) to (-0.33,0);
	\draw[green!75!black] (0.66,1) to (1,1) to (0.33,0);
	\end{scope}
	\begin{scope}[xshift=-0.33cm]
	\path[fill=red!10] (-1,-1) to (-0.66,-1) to (0,0) to (0.66,1) to (1,1) to (0.33,0) to (0,0) to (-0.33,0) to (-1,-1);
		\draw[red!75!black] (-0.66,-1) to (0,0) to (-0.33,0);
	\draw[red!75!black] (0.66,1) to (1,1) to (0.33,0);
	\end{scope}
	\begin{scope}[xshift=-0.66cm]
	\path[fill=blue!10] (-1,-1) to (-0.66,-1) to (0,0) to (0.66,1) to (1,1) to (0.33,0) to (0,0) to (-0.33,0) to (-1,-1);
	\draw[blue!75!black] (-0.66,-1) to (0,0) to (-0.33,0);
	\draw[blue!75!black] (0.66,1) to (1,1) to (0.33,0);
	\end{scope}
		\begin{scope}[xshift=-1cm]
	\path[fill=ggreen!10] (-1,-1) to (-0.66,-1) to (0,0) to (0.66,1) to (1,1) to (0.33,0) to (0,0) to (-0.33,0) to (-1,-1);
	\draw[green!75!black] (-0.66,-1) to (0,0) to (-0.33,0);
		\draw[green!75!black] (-0.66,-1) to (0,0) to (-0.33,0);
	\draw[green!75!black] (0.66,1) to (1,1) to (0.33,0);
	\end{scope}
		\begin{scope}[xshift=-1.33cm]
	\path[fill=red!10] (-1,-1) to (-0.66,-1) to (0,0) to (0.66,1) to (1,1) to (0.33,0) to (0,0) to (-0.33,0) to (-1,-1);
		\draw[red!75!black] (-0.66,-1) to (0,0) to (-0.33,0);
	\draw[red!75!black] (0.66,1) to (1,1) to (0.33,0);
	\end{scope}
	\begin{scope}[xshift=-1.66cm]
	\path[fill=blue!10] (-1,-1) to (-0.66,-1) to (0,0) to (0.66,1) to (1,1) to (0.33,0) to (0,0) to (-0.33,0) to (-1,-1);
	\draw[blue!75!black] (-0.66,-1) to (0,0) to (-0.33,0);
	\draw[blue!75!black] (0.66,1) to (1,1) to (0.33,0);
	\end{scope}
		\begin{scope}[xshift=-2cm]
	\path[fill=ggreen!10] (-1,-1) to (-0.66,-1) to (0,0) to (0.66,1) to (1,1) to (0.33,0) to (0,0) to (-0.33,0) to (-1,-1);
		\draw[green!75!black] (-0.66,-1) to (0,0) to (-0.33,0);
	\draw[green!75!black] (0.66,1) to (1,1) to (0.33,0);
	\end{scope}
			\begin{scope}[xshift=-2.33cm]
	\path[fill=red!10] (-1,-1) to (-0.66,-1) to (0,0) to (0.66,1) to (1,1) to (0.33,0) to (0,0) to (-0.33,0) to (-1,-1);
		\draw[red!75!black] (-0.66,-1) to (0,0) to (-0.33,0);
	\draw[red!75!black] (0.66,1) to (1,1) to (0.33,0);
	\end{scope}
	\begin{scope}[xshift=1.66cm,yshift=-2cm]
	\path[fill=red!10] (-1,-1) to (-0.66,-1) to (0,0) to (0.66,1) to (1,1) to (0.33,0) to (0,0) to (-0.33,0) to (-1,-1);
		\draw[red!75!black] (-0.66,-1) to (0,0) to (-0.33,0);
	\draw[red!75!black] (0.66,1) to (1,1) to (0.33,0);
	\end{scope}
	\begin{scope}[xshift=1.33cm,yshift=-2cm]
	\path[fill=blue!10] (-1,-1) to (-0.66,-1) to (0,0) to (0.66,1) to (1,1) to (0.33,0) to (0,0) to (-0.33,0) to (-1,-1);
	\draw[blue!75!black] (-0.66,-1) to (0,0) to (-0.33,0);
	\draw[blue!75!black] (0.66,1) to (1,1) to (0.33,0);
	\end{scope}
	\begin{scope}[xshift=1cm,yshift=-2cm]
	\path[fill=green!10] (-1,-1) to (-0.66,-1) to (0,0) to (0.66,1) to (1,1) to (0.33,0) to (0,0) to (-0.33,0) to (-1,-1);
		\draw[green!75!black] (-0.66,-1) to (0,0) to (-0.33,0);
	\draw[green!75!black] (0.66,1) to (1,1) to (0.33,0);
	\end{scope}
	\begin{scope}[xshift=0.66cm,yshift=-2cm]
	\path[fill=red!10] (-1,-1) to (-0.66,-1) to (0,0) to (0.66,1) to (1,1) to (0.33,0) to (0,0) to (-0.33,0) to (-1,-1);
		\draw[red!75!black] (-0.66,-1) to (0,0) to (-0.33,0);
	\draw[red!75!black] (0.66,1) to (1,1) to (0.33,0);
	\end{scope}
	\begin{scope}[xshift=0.33cm,yshift=-2cm]
	\path[fill=blue!10] (-1,-1) to (-0.66,-1) to (0,0) to (0.66,1) to (1,1) to (0.33,0) to (0,0) to (-0.33,0) to (-1,-1);
	\draw[blue!75!black] (-0.66,-1) to (0,0) to (-0.33,0);
	\draw[blue!75!black] (0.66,1) to (1,1) to (0.33,0);
	\end{scope}
	\begin{scope}[yshift=-2cm]  
	\path[fill=green!10] (-1,-1) to (-0.66,-1) to (0,0) to (0.66,1) to (1,1) to (0.33,0) to (0,0) to (-0.33,0) to (-1,-1);	
		\draw[green!75!black] (-0.66,-1) to (0,0) to (-0.33,0);
	\draw[green!75!black] (0.66,1) to (1,1) to (0.33,0);
	\end{scope}
	\begin{scope}[xshift=-0.33cm,yshift=-2cm]
	\path[fill=red!10] (-1,-1) to (-0.66,-1) to (0,0) to (0.66,1) to (1,1) to (0.33,0) to (0,0) to (-0.33,0) to (-1,-1);
		\draw[red!75!black] (-0.66,-1) to (0,0) to (-0.33,0);
	\draw[red!75!black] (0.66,1) to (1,1) to (0.33,0);
	\end{scope}
	\begin{scope}[xshift=-0.66cm,yshift=-2cm]
	\path[fill=blue!10] (-1,-1) to (-0.66,-1) to (0,0) to (0.66,1) to (1,1) to (0.33,0) to (0,0) to (-0.33,0) to (-1,-1);
	\draw[blue!75!black] (-0.66,-1) to (0,0) to (-0.33,0);
	\draw[blue!75!black] (0.66,1) to (1,1) to (0.33,0);
	\end{scope}
		\begin{scope}[xshift=-1cm,yshift=-2cm]
	\path[fill=ggreen!10] (-1,-1) to (-0.66,-1) to (0,0) to (0.66,1) to (1,1) to (0.33,0) to (0,0) to (-0.33,0) to (-1,-1);
		\draw[green!75!black] (-0.66,-1) to (0,0) to (-0.33,0);
	\draw[green!75!black] (0.66,1) to (1,1) to (0.33,0);
	\end{scope}
		\begin{scope}[xshift=-1.33cm,yshift=-2cm]
	\path[fill=red!10] (-1,-1) to (-0.66,-1) to (0,0) to (0.66,1) to (1,1) to (0.33,0) to (0,0) to (-0.33,0) to (-1,-1);
		\draw[red!75!black] (-0.66,-1) to (0,0) to (-0.33,0);
	\draw[red!75!black] (0.66,1) to (1,1) to (0.33,0);
	\end{scope}
	\begin{scope}[xshift=-1.66cm,yshift=-2cm]
	\path[fill=blue!10] (-1,-1) to (-0.66,-1) to (0,0) to (0.66,1) to (1,1) to (0.33,0) to (0,0) to (-0.33,0) to (-1,-1);
	\draw[blue!75!black] (-0.66,-1) to (0,0) to (-0.33,0);
	\draw[blue!75!black] (0.66,1) to (1,1) to (0.33,0);
	\end{scope}
	\begin{scope}[xshift=-2cm,yshift=-2cm]
	\path[fill=ggreen!10] (-1,-1) to (-0.66,-1) to (0,0) to (0.66,1) to (1,1) to (0.33,0) to (0,0) to (-0.33,0) to (-1,-1);
		\draw[green!75!black] (-0.66,-1) to (0,0) to (-0.33,0);
	\draw[green!75!black] (0.66,1) to (1,1) to (0.33,0);
	\end{scope}
	\begin{scope}[xshift=-2.33cm,yshift=-2cm]
	\path[fill=red!10] (-1,-1) to (-0.66,-1) to (0,0) to (0.66,1) to (1,1) to (0.33,0) to (0,0) to (-0.33,0) to (-1,-1);
		\draw[red!75!black] (-0.66,-1) to (0,0) to (-0.33,0);
	\draw[red!75!black] (0.66,1) to (1,1) to (0.33,0);
	\end{scope}

	\begin{scope}[yshift=0cm]  
	\draw[black!90,dashed,line width=0.3mm] (-1,0) to (-1,-1) to (0,-1);
	\draw[black!90,line width=0.1mm] (-1,0) to (0,0) to (0,-1);
	\end{scope}

	\node[sgreen dot] at (0,0) {};
	\node[sgreen dot] at (1,1) {};
	\node[sgreen dot] at (1,0) {};
	\node[sgreen dot] at (1,-1) {};
	\node[sgreen dot] at (0,-1) {};
	\node[sgreen dot] at (0,-1) {};
	\node[sgreen dot] at (0,1) {};
	\node[sgreen dot] at (-1,0) {};
	\node[sgreen dot] at (-1,1) {};
	\node[sgreen dot] at (-1,-1) {};
	
	\begin{scope}[xshift=0.33cm]
	 \node[sblue dot] at (0,0) {};
	\node[sblue dot] at (1,1) {};
	\node[sblue dot] at (1,0) {};
	\node[sblue dot] at (1,-1) {};
	\node[sblue dot] at (0,-1) {};
	\node[sblue dot] at (0,-1) {};
	\node[sblue dot] at (0,1) {};
	\node[sblue dot] at (-1,0) {};
	\node[sblue dot] at (-1,1) {};
	\node[sblue dot] at (-1,-1) {};
	\end{scope}
	
	\begin{scope}[xshift=-0.33cm]
	 \node[sred dot] at (0,0) {};
	\node[sred dot] at (1,1) {};
	\node[sred dot] at (1,0) {};
	\node[sred dot] at (1,-1) {};
	\node[sred dot] at (0,-1) {};
	\node[sred dot] at (0,-1) {};
	\node[sred dot] at (0,1) {};
	\node[sred dot] at (-1,0) {};
	\node[sred dot] at (-1,1) {};
	\node[sred dot] at (-1,-1) {};
	\end{scope}

	\node[ggreen] at (-0.51,-.5) {$\Delta_0$};
	\node[red] at (-0.84,-0.5) {$\Delta_1$};
	\node[blue] at (-0.18,-0.5) {$\Delta_2$};
\end{tikzpicture}
\caption{The fundamental domain $(-1,0]^2$ for $m=3$, $k=2$ with chambers $\Delta_0$, $\Delta_1$, $\Delta_2$ .}
\end{figure}

The module 
$$\M=\bigoplus_{i=0}^{m-1}A_{i}$$
is a complete sum of conic modules as in \eqref{eqM}. The graded simples of $\End_R(\M)$ all
have projective dimension $2$, with projective resolution $\Hom_R(\M,\mathcal{K}^\bullet_{\Delta_{i}})$ coming from \eqref{conic-res-dim2}. This implies that $\End_R(\M)$ is an NCCR of $R$; in fact the only NCCR up to Morita equivalence, see \cite[Prop.~2.6]{IyamaWemyss-NCBondalOrlov}.
\end{ex}

\subsection{The non-simplicial case}

The preceding argument always fails in the non-simplicial case for the following geometric reason.

\begin{lemma}
If $R$ is not simplicial, then there exists a chamber of constancy $\Delta$ containing no $0$-cells.
\end{lemma}

\begin{proof}
We proceed by induction on $|\Sigma_1|$. For the base case, assume $|\Sigma_1|-1=d:=\mathrm{rank}(M)$ and enumerate the elements as $n_1,n_2,\ldots ,n_{d+1}$. Choose an identification $\det:\Lambda^dN\garrow{\sim} \mathbb{Z}$, and set 
\[ \det(\widehat{n_i}) := \det(n_1\wedge n_2\wedge \ldots \wedge \widehat{n_i}\wedge \ldots \wedge n_{d+1}) \]
The identity $n_1\wedge n_2\wedge \ldots \wedge {n_i}\wedge \ldots \wedge n_{d+1} =0$ in the exterior algebra of $N$ implies 
\begin{equation}\label{eq: Cramer}
\sum_{i=1}^{d+1} (-1)^i \det(\widehat{n_i}) n_i =0 \ . 
\end{equation}

Pick an $i$ such that $|\mathrm{det}(\widehat{n_i})|$ is maximized; in particular, $\mathrm{det}(\widehat{n_i})\neq0$. 
First, we claim there is a chamber of constancy $\Delta$ given by
\[ \Delta = \{ x\in M_\mathbb{R} \mid \delta_{ij}-1 < \langle x,n_j \rangle \leq \delta_{ij} \text{ for all }j=1,\ldots,d+1\} \]
where $\delta_{ij}$ is the Kronecker $\delta$.\footnote{Equivalently, the negative standard basis vector $-e_i$ is a conic divisor; see Remark \ref{ConeList}.} 
To see this, consider a point $w$ in the interior of the facet of $C$ with normal $n_i$. This implies that $\langle w,n_j\rangle\geq0$, with equality if and only if $i=j$. For sufficiently small $\varepsilon>0$, we have $-1< \langle -\varepsilon w,n_j\rangle\leq 0$, with equality if and only if $i=j$. Adding a small perturbation $p$, we find a point $-\varepsilon w+p$ with
\[ \lceil\langle -\varepsilon w + p, n_j \rangle \rceil = \delta_{ij} \text{ for all }j=1,\ldots ,d+1 \]
and so $-\varepsilon w+p$ is in the desired chamber of constancy $\Delta$.

Now, assume that $\Delta$ does have a $0$-cell, at a point $v\in M_\mathbb{R}$. This means $v$ is in $d$-many linearly independent hyperplanes, or equivalently, we can find a $j\in \{1,2,\ldots,d+1\}$ such that $\det(\widehat{n_j})\neq0$ and
\[ \langle v,n_k\rangle = \lceil \langle v,n_k\rangle \rceil  \text{ for all $k\neq j$} \ . \]
By the defining equations of $\Delta$, $\lceil \langle v,n_k\rangle \rceil =1 $ if $k=i$ and $0$ otherwise.
Therefore, applying $\langle v,-\rangle$ to \eqref{eq: Cramer} yields
\[ 0 = \sum_{k=1}^{d+1} (-1)^k \det(\widehat{n_k}) \langle v, n_k \rangle = (-1)^i\det(\widehat{n_i}) + (-1)^j \det(\widehat{n_j})\langle v,n_j\rangle \ . \]
Since $\det(\widehat{n_i}) \neq 0$, this equation implies that both $\det(\widehat{n_j})$ and $\langle v, n_j \rangle$ are non-zero. Solving for $\langle v,n_j\rangle$, we find
\[ \langle v, n_j\rangle = (-1)^{j-i+1}\frac{\det(\widehat{n_i})}{\det(\widehat{n_j})} \ , \]
so rounding up, we have $\langle v, n_j \rangle \neq 0$. 
This contradicts the assumption that $v$ was in the chamber $\Delta$. Therefore, $\Delta$ has no $0$-cells.

Next, assume the proposition is true whenever $|\Sigma_1|=i$ for some $i>d$. Consider a $C$ with $|\Sigma_1|=i+1$, and enumerate the elements $n_1,n_2,\ldots,n_{i+1}$. Since $C$ is pointed, the set of normals is a spanning set, and so there is an $i$-tuple of them which is also a spanning set. Without loss of generality, assume $n_1,n_2,\ldots,n_i$ is a spanning set, and let $C'$ be the cone they define.

Consider the chambers of constancy associated to the cone $C'$. By the inductive hypothesis, there is a chamber $\Delta'$ which does not contain a $0$-cell. The chamber $\Delta'$ is a union of chambers of constancy of $C$, each corresponding to a different value of $\lceil \langle -,n_{i+1}\rangle\rceil$. Let $\Delta$ be the chamber of $C$ contained in $\Delta'$ on which $\lceil \langle -,n_{i+1}\rangle\rceil$ is maximal.

A $0$-cell of $\Delta$ would have to be the intersection of a $1$-cell in $\Delta'$ with a hyperplane of the form $\langle  v,n_{i+1}\rangle =a$ for some $a\in \mathbb{Z}$. However, if the $1$-cell is transverse to the hyperplane, then there are points in $\Delta'$ on which $\langle  v,n_{i+1}\rangle >a$, contradicting maximality. If the $1$-cell is contained in the hyperplane, then their intersection is not a point. Therefore, $\Delta$ does not contain a $0$-cell, completing the induction.
\end{proof}

\begin{prop}\label{nonsimplicialnonNCCR}
If $R$ is a non-simplicial toric algebra and $\M$ is a complete sum of conic modules, then $\mathrm{End}_R(\M)$ is not an NCCR of $R$.
\end{prop}
\begin{proof}
By the preceding lemma, there is a chamber of constancy $\Delta$ with no $0$-cell. By Corollary \ref{pdim}, the simple $\mathrm{End}_R(\M)$-module $S_\Delta$ has projective dimension strictly less than $\dim(R)$. Hence, not every simple $\mathrm{End}_R(\M)$-module has the same projective dimension, and so it is not an NCCR.
\end{proof}

\begin{ex}
Look again at the cone over the square in $\RR^3$ from Examples \ref{non-simplicial-conic-complex} and \ref{non-simplicial}. Here $\End_R(\M)$, for $\M=A_0\oplus A_1 \oplus A_2$,  is of global dimension $3$ but has two simples of projective dimension $2$, coming from the two tetrahedral chambers of constancy. This reproves the fact that $\End_R(\M)$ is not an NCCR.\footnote{See Chapter 5 of \cite{Quarles} or Theorem P.2 of \cite{Leuschke}.}

However, this can be fixed by omitting one of the two tetrahedral conic modules. For example, if we omit $A_2$, we can build analogous conic complexes 
\begin{align*}
 A_0 \longrightarrow A_1\oplus A_0^{\oplus 4}  & \longrightarrow  A_0^{\oplus4}\oplus A_1^{\oplus2}  \longrightarrow A_0 \ , \\
 A_1 \longrightarrow A_0^{\oplus2} & \longrightarrow A_0^{\oplus2} \longrightarrow A_1 \ , 
\end{align*}
which the functor $\mathrm{Hom}_R(A_0\oplus A_1,-)$ takes to minimal length projective resolutions of the only two graded simple $\mathrm{End}_R(A_0\oplus A_1)$-modules.
By an analogous argument to the simplicial case, this implies $\End_R(A_0\oplus A_1)$ is an NCCR of $R$.
\end{ex}

%
%
%
%
%

\section{Acknowledgements} The authors are grateful to Colin Ingalls for his interest and stimulating discussions while he was visiting Ann Arbor. The first author wishes to thank
Michael Wemyss for helpful discussions, as well, and the third author Kevin Tucker and Michel Van den Bergh. We also want to thank the referee for valuable comments and suggestions.


\begin{thebibliography}{R{\v S}VdB17}

\bibitem[Aus55]{AuslanderDimensionIII}
M.~Auslander.
\newblock On the dimension of modules and algebras. {III}. {G}lobal dimension.
\newblock {\em Nagoya Math. J.}, 9:67--77, 1955.

\bibitem[Bae04]{Baetica}
Cornel Baetica.
\newblock Cohen-{M}acaulay classes which are not conic.
\newblock {\em Comm. Algebra}, 32(3):1183--1188, 2004.

\bibitem[Bas68]{Bas68}
Hyman Bass.
\newblock {\em Algebraic {$K$}-theory}.
\newblock W. A. Benjamin, Inc., New York-Amsterdam, 1968.

\bibitem[Ber58]{Berstein}
I.~Ber{\v{s}}te{\u\i}n.
\newblock On the dimension of modules and algebras. {IX}. {D}irect limits.
\newblock {\em Nagoya Math. J.}, 13:83--84, 1958.

\bibitem[BG03]{BrunsGubeladze}
Winfried Bruns and Joseph Gubeladze.
\newblock Divisorial linear algebra of normal semigroup rings.
\newblock {\em Algebr. Represent. Theory}, 6(2):139--168, 2003.

\bibitem[Bli13]{Blickle}
Manuel Blickle.
\newblock Test ideals via algebras of {$p^{-e}$}-linear maps.
\newblock {\em J. Algebraic Geom.}, 22(1):49--83, 2013.

\bibitem[BO02]{BondalOrlov}
A.~Bondal and D.~Orlov.
\newblock Derived categories of coherent sheaves.
\newblock In {\em Proceedings of the {I}nternational {C}ongress of
  {M}athematicians, {V}ol. {II} ({B}eijing, 2002)}, pages 47--56. Higher Ed.
  Press, Beijing, 2002.

\bibitem[Boc12]{BocklandtNCCR}
R.~Bocklandt.
\newblock Generating toric noncommutative crepant resolutions.
\newblock {\em J. Algebra}, 364:119--147, 2012.

\bibitem[Bri93]{Brion}
Michel Brion.
\newblock Sur les modules de covariants.
\newblock {\em Ann. Sci. \'{E}cole Norm. Sup. (4)}, 26(1):1--21, 1993.

\bibitem[Bro12]{Broomhead}
Nathan Broomhead.
\newblock Dimer models and {C}alabi-{Y}au algebras.
\newblock {\em Mem. Amer. Math. Soc.}, 215(1011):viii+86, 2012.

\bibitem[Bru05]{Bruns/arXiv:math/0408110}
Winfried Bruns.
\newblock Conic divisor classes over a normal monoid algebra.
\newblock In {\em Commutative algebra and algebraic geometry}, volume 390 of
  {\em Contemp. Math.}, pages 63--71. Amer. Math. Soc., Providence, RI, 2005.

\bibitem[Cha74]{Chase}
Stephen~U. Chase.
\newblock On the homological dimension of algebras of differential operators.
\newblock {\em Comm. Algebra}, 1:351--363, 1974.

\bibitem[CQV12]{Craw-Quintero-Velez}
Alastair Craw and Alexander Quintero~V\'{e}lez.
\newblock Cellular resolutions of noncommutative toric algebras from
  superpotentials.
\newblock {\em Adv. Math.}, 229(3):1516--1554, 2012.

\bibitem[CRST17]{C-RST}
J.~Carvajal-Rojas, K.~Schwede, and K.~Tucker.
\newblock {Bertini Theorems for $F$-signature}.
\newblock {\em preprint, {\tt arXiv:1710.01277}}, 2017.

\bibitem[D{\u a}s92]{Dascalescu}
Sorin D{\u a}sc{\u a}lescu.
\newblock Graded semiperfect rings.
\newblock {\em Bull. Math. Soc. Sci. Math. Roumanie (N.S.)},
  36(84)(3-4):247--254, 1992.

\bibitem[DFI15]{DaFI}
H.~Dao, E.~Faber, and C.~Ingalls.
\newblock Noncommutative (crepant) desingularizations and the global spectrum
  of commutative rings.
\newblock {\em Algebr. Represent. Theory}, 18(3):633--664, 2015.

\bibitem[DITV15]{NCR}
H.~Dao, O.~Iyama, R.~Takahashi, and C.~Vial.
\newblock Non-commutative resolutions and {G}rothendieck groups.
\newblock {\em J. Noncommut. Geom.}, 9(1):21--34, 2015.

\bibitem[Dix63]{Dix63}
J.~Dixmier.
\newblock Repr\'esentations irr\'eductibles des alg\`ebres de {L}ie
  nilpotentes.
\newblock {\em An. Acad. Brasil. Ci.}, 35:491--519, 1963.

\bibitem[Don02]{Dong}
X.~Dong.
\newblock Canonical modules of semigroup rings and a conjecture of {R}einer.
\newblock {\em Discrete Comput. Geom.}, 27(1):85--97, 2002.
\newblock Geometric combinatorics (San Francisco, CA/Davis, CA, 2000).

\bibitem[Eis95]{Eis95}
David Eisenbud.
\newblock {\em Commutative algebra. With a view toward algebraic geometry},
  volume 150 of {\em Graduate Texts in Mathematics}.
\newblock Springer-Verlag, New York, 1995.

\bibitem[Ful93]{Fulton}
William Fulton.
\newblock {\em Introduction to Toric Varieties}, volume 131 of {\em Annals of
  Mathematics Studies}.
\newblock Princeton University Press, Princeton, 1993.

\bibitem[HGK04]{Hazewinkeletc}
M.~Hazewinkel, N.~Gubareni, and V.~V. Kirichenko.
\newblock {\em Algebras, rings and modules. {V}ol. 1}, volume 575 of {\em
  Mathematics and its Applications}.
\newblock Kluwer Academic Publishers, Dordrecht, 2004.

\bibitem[HH90]{HH90}
Melvin Hochster and Craig Huneke.
\newblock Tight closure, invariant theory, and the {B}rian\c{c}on-{S}koda
  theorem.
\newblock {\em J. Amer. Math. Soc.}, 3(1):31--116, 1990.

\bibitem[HL02]{HunekeLeuschke}
Craig Huneke and Graham~J. Leuschke.
\newblock Two theorems about maximal {C}ohen-{M}acaulay modules.
\newblock {\em Math. Ann.}, 324(2):391--404, 2002.

\bibitem[HN17]{HibiRings}
A.~Higashitani and Y.~Nakajima.
\newblock {Conic divisorial ideals of Hibi rings and their applications to
  non-commutative crepant resolutions}.
\newblock {\em preprint, {\tt arXiv:1702.07058}}, 2017.

\bibitem[HW02]{HaraWatanabe}
Nobuo Hara and Kei-Ichi Watanabe.
\newblock F-regular and {F}-pure rings vs. log terminal and log canonical
  singularities.
\newblock {\em J. Algebraic Geom.}, 11(2):363--392, 2002.

\bibitem[HX15]{Hacon}
Christopher~D. Hacon and Chenyang Xu.
\newblock On the three dimensional minimal model program in positive
  characteristic.
\newblock {\em J. Amer. Math. Soc.}, 28(3):711--744, 2015.

\bibitem[HY03]{HaraYoshida}
Nobuo Hara and Ken-Ichi Yoshida.
\newblock A generalization of tight closure and multiplier ideals.
\newblock {\em Trans. Amer. Math. Soc.}, 355(8):3143--3174, 2003.

\bibitem[IW13]{IyamaWemyss-NCBondalOrlov}
Osamu Iyama and Michael Wemyss.
\newblock On the noncommutative {B}ondal-{O}rlov conjecture.
\newblock {\em J. Reine Angew. Math.}, 683:119--128, 2013.

\bibitem[IW14]{IyamaWemyss}
O.~Iyama and M.~Wemyss.
\newblock Maximal modifications and {A}uslander-{R}eiten duality for
  non-isolated singularities.
\newblock {\em Invent. Math.}, 197(3):521--586, 2014.

\bibitem[Jef17]{Jeff17}
Jack Jeffries.
\newblock Derived functors of differential operators.
\newblock {\em preprint, {\tt arXiv:1711.03960}}, 2017.

\bibitem[JNB17]{JeffriesNunez}
Jack Jeffries and Luis {N}\'{u}{\~{n}}ez Betancourt.
\newblock Quantifying singularities with differential operators.
\newblock 2017.
\newblock in preparation.

\bibitem[Kun69]{Kunz}
Ernst Kunz.
\newblock Characterizations of regular local rings for characteristic {$p$}.
\newblock {\em Amer. J. Math.}, 91:772--784, 1969.

\bibitem[Leu12]{Leuschke}
Graham~J. Leuschke.
\newblock Non-commutative crepant resolutions: scenes from categorical
  geometry.
\newblock In {\em Progress in commutative algebra 1}, pages 293--361. de
  Gruyter, Berlin, 2012.

\bibitem[Mun84]{Munkres}
James~R. Munkres.
\newblock {\em Elements of algebraic topology}.
\newblock Addison-Wesley Publishing Company, Menlo Park, CA, 1984.

\bibitem[MVdB98]{MussonVdB}
Ian~M. Musson and Michel Van~den Bergh.
\newblock Invariants under tori of rings of differential operators and related
  topics.
\newblock {\em Mem. Amer. Math. Soc.}, 136(650):viii+85, 1998.

\bibitem[NVO79]{GradedRings}
Constantin N{\u a}st{\u a}sescu and F.~Van~Oystaeyen.
\newblock {\em Graded and filtered rings and modules}, volume 758 of {\em
  Lecture Notes in Mathematics}.
\newblock Springer, Berlin, 1979.

\bibitem[Qua05]{Quarles}
C.~Quarles.
\newblock {Krull--Schmidt rings and noncommutative resolutions of
  singularities}.
\newblock Master's thesis, University of Washington, 2005.

\bibitem[Roo72]{Roos}
Jan-Erik Roos.
\newblock D\'etermination de la dimension homologique globale des alg\`ebres de
  {W}eyl.
\newblock {\em C. R. Acad. Sci. Paris S\'er. A-B}, 274:A23--A26, 1972.

\bibitem[R{\v S}VdB17]{RSvdB3}
T.~Raedschelders, {\v S}.~{\v S}penko, and M.~Van~den Bergh.
\newblock {The Frobenius morphism in invariant theory}.
\newblock {\em preprint, {\tt arXiv:1705.01832}}, 2017.

\bibitem[Sch09]{Schwede}
Karl Schwede.
\newblock {$F$}-injective singularities are {D}u {B}ois.
\newblock {\em Amer. J. Math.}, 131(2):445--473, 2009.

\bibitem[Smi87]{PSmithCharpRegular}
S.~P. Smith.
\newblock The global homological dimension of the ring of differential
  operators on a nonsingular variety over a field of positive characteristic.
\newblock {\em J. Algebra}, 107(1):98--105, 1987.

\bibitem[Smi95]{SmithDMod}
Karen~E. Smith.
\newblock The {$D$}-module structure of {$F$}-split rings.
\newblock {\em Math. Res. Lett.}, 2(4):377--386, 1995.

\bibitem[Smi97]{Smi97b}
Karen~E. Smith.
\newblock Vanishing, singularities and effective bounds via prime
  characteristic local algebra.
\newblock In {\em Algebraic geometry---{S}anta {C}ruz 1995}, volume~62 of {\em
  Proc. Sympos. Pure Math.}, pages 289--325. Amer. Math. Soc., Providence, RI,
  1997.

\bibitem[Smi00]{SmithGloballyFRegular}
Karen~E. Smith.
\newblock Globally {F}-regular varieties: applications to vanishing theorems
  for quotients of {F}ano varieties.
\newblock {\em Michigan Math. J.}, 48:553--572, 2000.
\newblock Dedicated to William Fulton on the occasion of his 60th birthday.

\bibitem[ST12]{ST11}
Karl Schwede and Kevin Tucker.
\newblock A survey of test ideals.
\newblock In {\em Progress in commutative algebra 2}, pages 39--99. Walter de
  Gruyter, Berlin, 2012.

\bibitem[Sta82]{Stanley}
Richard~P. Stanley.
\newblock Linear {D}iophantine equations and local cohomology.
\newblock {\em Invent. Math.}, 68(2):175--193, 1982.

\bibitem[SVdB97]{SmVdB}
Karen~E. Smith and Michel Van~den Bergh.
\newblock Simplicity of rings of differential operators in prime
  characteristic.
\newblock {\em Proc. London Math. Soc. (3)}, 75(1):32--62, 1997.

\bibitem[SVdB08]{StaffordVanden}
J.~T. Stafford and M.~Van~den Bergh.
\newblock Noncommutative resolutions and rational singularities.
\newblock {\em Michigan Math. J.}, 57:659--674, 2008.
\newblock Special volume in honor of Melvin Hochster.

\bibitem[{\v S}VdB17a]{SVdB1}
{\v S}.~{\v S}penko and M.~Van~den Bergh.
\newblock {Non-commutative resolutions of some toric varieties I}.
\newblock {\em preprint, {\tt arXiv:1701.05255}}, 2017.

\bibitem[{\v S}VdB17b]{SVdB2}
{\v S}pela {\v S}penko and Michel Van~den Bergh.
\newblock Non-commutative resolutions of quotient singularities for reductive
  groups.
\newblock {\em Invent. Math.}, 210(1):3--67, 2017.

\bibitem[Tak04]{Takagi}
Shunsuke Takagi.
\newblock An interpretation of multiplier ideals via tight closure.
\newblock {\em J. Algebraic Geom.}, 13(2):393--415, 2004.

\bibitem[Tuc12]{Tucker}
Kevin Tucker.
\newblock {$F$}-signature exists.
\newblock {\em Invent. Math.}, 190(3):743--765, 2012.

\bibitem[VdB93]{VdB-CM-tori}
Michel Van~den Bergh.
\newblock Cohen-{M}acaulayness of semi-invariants for tori.
\newblock {\em Trans. Amer. Math. Soc.}, 336(2):557--580, 1993.

\bibitem[VdB04]{VanDenBergh}
Michel Van~den Bergh.
\newblock Non-commutative crepant resolutions.
\newblock In {\em The legacy of {N}iels {H}enrik {A}bel}, pages 749--770.
  Springer, Berlin, 2004.

\bibitem[Wem16]{Wemyss}
Michael Wemyss.
\newblock Noncommutative resolutions.
\newblock In {\em Noncommutative algebraic geometry}, volume~64 of {\em Math.
  Sci. Res. Inst. Publ.}, pages 239--306. Cambridge Univ. Press, New York,
  2016.

\bibitem[Yas10]{Yasuda}
T.~Yasuda.
\newblock {Noncommutative resolution of toric singularities: An application of
  Frobenius morphism of noncommutative blowup}.
\newblock {\em preprint, {\tt arXiv:1002.0181v1}}, 2010.

\bibitem[Yek92]{Yekutieli}
Amnon Yekutieli.
\newblock An explicit construction of the {G}rothendieck residue complex.
\newblock {\em Ast\'erisque}, (208):127, 1992.
\newblock With an appendix by Pramathanath Sastry.

\end{thebibliography}

 \def\cprime{$'$}

\end{document}